\documentclass[a4paper,12pt]{article}
\usepackage[utf8]{inputenc}

\usepackage{amssymb}
\usepackage{amsmath}
\usepackage{amsthm}
\usepackage{tikz}
\usepackage{pgfplots}
\pgfplotsset{compat=1.8}
\usepackage{tikz-cd}
\usepackage{hyperref}
\usepackage{cleveref}
\usepackage{listings}
\usepackage[maxnames=50]{biblatex}
\title{Intermediate hyperbolicity of varieties supporting a variation of Hodge structure}
\author{\'Eloan Rapion}
\date{}

\AtBeginBibliography{\small}

\newcommand{\dif}{\mathop{}\mathopen{}\mathrm d}

\newcommand*{\application}[5]{\begin{array}{lrcl}
		#1: & #2 & \to & #3 \\
		& #4 & \mapsto & #5
	\end{array}
}
\newcommand*{\app}[4]{\begin{array}{rcl}
	#1 & \to & #2 \\
	#3 & \mapsto & #4
\end{array}
}
\newcommand{\ext}[1]%
{{\vphantom{#1}}^{\mathit \diamond}{#1}}

\newtheorem{thm}{Theorem}[section]

\newtheorem{defi}[thm]{Definition}

\newtheorem{exs}[thm]{Examples}
\newtheorem{lem}[thm]{Lemma}
\newtheorem{prop}[thm]{Proposition}
\newtheorem*{ack}{Acknowledgments}
\theoremstyle{remark}
\newtheorem{nota}[thm]{Notation}
\newtheorem{rem}[thm]{Remark}
\newtheorem{rems}[thm]{Remarks}

\crefname{thm}{Theorem}{Theorems}
\Crefname{thm}{Theorem}{Theorems}
\crefname{cor}{Corollary}{Corollaries}
\Crefname{cor}{Corollary}{Corollaries}
\crefname{defi}{Definition}{Definitions}
\Crefname{defi}{Definition}{Definitions}
\crefname{defis}{Definitions}{Definitions}
\Crefname{defis}{Definitions}{Definitions}
\crefname{ex}{Example}{Examples}
\Crefname{ex}{Example}{Examples}
\crefname{exs}{Examples}{Examples}
\Crefname{exs}{Examples}{Examples}
\crefname{lem}{Lemma}{Lemmas}
\Crefname{lem}{Lemma}{Lemmas}
\crefname{prop}{Proposition}{Propositions}
\Crefname{prop}{Proposition}{Propositions}
\crefname{nota}{Notation}{Notations}
\Crefname{nota}{Notation}{Notations}
\crefname{rem}{Remark}{Remarks}
\Crefname{rem}{Remark}{Remarks}
\crefname{rems}{Remarks}{Remarks}
\Crefname{rems}{Remarks}{Remarks}

\DeclareMathOperator{\Ad}{Ad}
\DeclareMathOperator{\AugBL}{\mathbb{B}_+}

\DeclareMathOperator{\Aut}{Aut}

\DeclareMathOperator{\Charac}{\mathcal{S}}

\DeclareMathOperator{\Deg}{Deg}
\DeclareMathOperator{\End}{End}

\DeclareMathOperator{\GL}{GL}
\DeclareMathOperator{\Gr}{\mathbf{Gr}}
\DeclareMathOperator{\Hol}{Hol}
\DeclareMathOperator{\Hom}{Hom}

\DeclareMathOperator{\rk}{rk}
\DeclareMathOperator{\PGL}{PGL}
\DeclareMathOperator{\Proj}{\mathbb{P}}

\DeclareMathOperator{\StBL}{\mathbb{B}}
\DeclareMathOperator{\Sym}{Sym}

\renewcommand{\epsilon}{\varepsilon}

\def\ad{\text{ad}}

\def\Amp{A}
\def\Anman{M}

\def\C{\mathbb{C}}
\def\Came{g}
\def\came{\hat{\Came}}
\def\Cat{\mathcal{C}}
\def\Charac{\mathcal{S}}

\def\Comp{\bar{\Qpv}}

\def\di{d}
\def\Disk{\Delta}

\def\Div{D}

\def\Dom{\Sigma}
\def\Dual{\check{\Dom}}

\def\glie{\mathfrak{g}}
\def\Group{G}
\def\hlie{\mathfrak{h}}
\def\Hprim{H_{\text{prim}}}

\def\I{\mathcal{I}}
\def\id{\text{id}}

\def\Isot{K}
\def\Latt{\Gamma}
\def\LB{L}
\def\llie{\mathfrak{l}}

\def\Lsav{X}

\def\mlie{\mathfrak{m}}
\def\N{\mathbb{N}}
\def\Ns{\N_{>0}}

\def\plie{\mathfrak{p}}

\def\Prv{W}
\def\Q{\mathbb{Q}}

\def\Qpv{V}
\def\R{\mathbb{R}}

\def\Ranexp{\R_{\text{an, exp}}}
\def\reg{0}
\def\Ricu{R}
\def\Siegel{\mathfrak{S}}

\def\sl{\mathfrak{sl}}

\def\Stru{\mathcal{O}}
\def\su{\mathfrak{su}}
\def\Taut{\Stru(1)}
\def\Tautp{\Stru_p(1)}
\def\Tautpp{\Stru'_p(1)}

\def\tV{\tilde{V}}
\def\U{\mathbb{U}}
\def\V{\mathbb{V}}

\def\VB{E}

\def\Z{\mathbb{Z}}

\begin{document}

\maketitle

\begin{abstract}
	Let $\Comp$ be a connected smooth complex projective variety. Let $\Div \subset \Comp$ be a normal crossing divisor. Let $\V$ be a complex polarizable variation of Hodge structure on $\Qpv := \Comp \setminus \Div$. Suppose that the period map of $\V$ is immersive at a point of $\Qpv$.
	
	We prove that for every integer $p$ with $1 \leq p \leq \dim \Qpv$, the vector bundle $\Omega_{\Comp}^p(\log \Div)$ is L-big (i.e. the tautological line bundle on $\Proj\Omega_{\Comp}^p(\log\Div)$ is big). If the local monodromy is quasi-unipotent, we give a method to determine an $m \in \N$ such that if $p > m$, then $\Omega_{\Comp}^p(\log \Div)$ is moreover Viehweg-big. We give the optimal value of $m$ explicitly when $\Qpv$ is a locally symmetric variety. We prove that if $\Qpv$ is a finite étale cover of the fine moduli space of smooth quintic threefolds, the result holds for $m = 90$ (in this case $\dim \Qpv = 101$).
	
	The proof of the previous results is based on a study of an augmented base locus associated with $\Omega_{\Comp}^p(\log \Div)$. In the case of locally symmetric varieties, we introduce ``higher degree characteristic subvarieties'', generalizing the characteristic subvariety defined by Mok in the case $p = 1$, and prove that they coincide with these augmented base loci.
\end{abstract}

\section{Introduction}

If a smooth complex quasiprojective variety admits a polarizable variation of Hodge structure, then the logarithmic cotangent bundle and the logarithmic canonical bundle on a smooth projective logarithmic compactification are known to have some positivity properties (see \cite{zuo2000negativity, brunebarbe2013symmetric, brunebarbe2020hyperbolicity, brunebarbe2018symmetric}). In this article, we investigate the positivity of the vector bundles of logarithmic differential forms of every degree.

\subsection{Bigness}

A first notion of positivity of vector bundles is the L-bigness. A vector bundle $\VB$ on an irreducible proper complex algebraic variety $\Prv$ is said to be \emph{L-big} if there exists $\epsilon \in \R_{>0}$ such that for every $k \gg 1$, we have $\dim H^0(\Prv,\Sym^k\VB) \geq \epsilon k^{\dim(\Prv)+\rk(\VB)-1}$. In particular, a line bundle is L-big if and only if it is big.
\begin{thm}\label{thm:lbig}
	Let $\Comp$ be a connected smooth complex projective variety. Let $\Div \subset \Comp$ be a normal crossing divisor. Let $\V$ be a complex polarizable variation of Hodge structure ($\C$-PVHS) on $\Qpv := \Comp \setminus \Div$. Let $p$ be an integer such that $1 \leq p \leq \dim \Qpv$. Suppose there exists a point of $\Qpv$ at which the period map of $\V$ is immersive.
	
	Then the vector bundle $\Omega_{\Comp}^p(\log \Div)$ is L-big.
\end{thm}
The cases of $p = 1$ and $p = \dim \Qpv$ of the previous theorem were proven by Brunebarbe and Cadorel \cite{brunebarbe2020hyperbolicity} (generalizing results of Zuo \cite{zuo2000negativity}, Brunebarbe-Klingler-Totaro \cite{brunebarbe2013symmetric} and Brunebarbe \cite{brunebarbe2018symmetric}). A similar proof gives the general case.

\bigskip

A stronger notion of positivity is the Viehweg-bigness. A vector bundle $\VB$ on an irreducible complex projective variety $\Prv$ is said to be \emph{ample with respect to} a Zariski open subset $U \subset \Prv$ if there exists an ample line bundle $\Amp$ on $\Prv$, two integers $a, b \in \Ns$ and a morphism $\Amp^{\oplus a} \to \Sym^b\VB$ that is surjective over $U$. If $U \neq \emptyset$, then $\VB$ is said to be \emph{Viehweg-big}. This notion is stronger than L-bigness, and is equivalent to it if $\VB$ is a line bundle. In particular, in the setting of \Cref{thm:lbig}, $\Omega_{\Comp}^{\dim \Comp}(\log \Div)$ is Viehweg-big. We prove that it is actually the case for every degree high enough.

\bigskip

To give a degree from which $\Omega_{\Comp}^{p}(\log \Div)$ is Viehweg-big, we consider some restrictions of iterations of the Higgs field. For a $\C$-PVHS $\V$ on a smooth complex algebraic variety $\Qpv$ and a point $x \in \Qpv$, we have a grading $\V_x = \oplus_{i \in \Z} \V_x^i$. For every $i, j \in \Z$ with $j \leq i$, we can restrict the $(i-j)$-th iterate $\theta^{i-j}$ of the Higgs field $\theta$ associated with $\V$ to define a map $\theta_x^{i,j}: T_{\Qpv,x} \to \Hom_\C(\V_x^i, \V_x^j)$.

\bigskip

In order to simplify the statement, we add here a ``Torelli type'' hypothesis $(*)$ and a ``Calabi-Yau type'' hypothesis $(**)$ (for the general statement without these hypotheses, see \Cref{thm:general}).

\begin{thm}\label{thm:vbig}
	Let $\Comp$ be a connected smooth complex projective variety. Let $\Div \subset \Comp$ be a normal crossing divisor. Let $\V$ be a $\C$-PVHS on $\Qpv := \Comp \setminus \Div$. Suppose $\V$ has quasi-unipotent local monodromy (around every irreducible component of $\Div$). Let $x \in \Qpv$ be a point where the period map of $\V$ is immersive. Let $i \in \Z$ be the biggest integer such that $\theta_x^{i,i-1}$ is nonzero. Suppose that:

	$(*)$ $\theta_x^{i,i-1}$ is injective,

	$(**)$ $\dim \V_x^i = 1$.

	Let $j$ be the smallest integer with $j < i$ such that $\theta_x^{i,j}$ is nonzero. Let $m$ be the maximum dimension for a vector subspace $S \subset T_{\Qpv,x}$ such that for every $v \in S$, $\theta_x^{i,j}(v) = 0$. Let $p \in \N$ be such that $m < p \leq \dim \Qpv$.
	
	Then $\Omega_{\Comp}^p(\log \Div)$ is Viehweg-big.
\end{thm}

\begin{rems}
	\begin{itemize}
		\item The hypotheses $(*)$ and $(**)$ are close to ``PVHS of type II'' defined by Sheng and Zuo in \cite[Definition 1.1]{sheng2010polarized}. The difference is that they require the bijectivity of $\theta_x^{i,i-1}$, at every point $x \in \Qpv$.
		\item Let $\Comp$ be a connected smooth complex projective variety, $\Div \subset \Comp$ a normal crossing divisor, $\V$ a $\C$-PVHS on $\Qpv := \Comp \setminus \Div$ with quasi-unipotent local monodromy, and $x \in \Qpv$ a point where the period map of $\V$ is immersive. Then the $\C$-PVHS $\bigotimes_{i = a}^b \Lambda^{r_i} \V$ verifies the hypotheses $(*)$ and $(**)$, where $a \leq b$ are two integers such that $\V_x = \oplus_{i = a}^b \V_x^i$ and $r_i := \sum_{j = i}^b \dim \V_x^j$.
	\end{itemize}
\end{rems}

The previous theorem applies to the moduli space of quintic threefolds. In the following theorem, the ``$91$'' stems from the limits of our computing power and could probably be improved.

\begin{thm}\label{thm:quintthree}
	Consider $\mathcal{M}$ the fine moduli space of smooth complex quintic threefolds, that is a smooth Deligne-Mumford stack that admits a finite étale cover $M \twoheadrightarrow \mathcal{M}$ with $M$ a connected smooth complex quasiprojective variety. Let $(\bar{M}, \Div)$ be a smooth projective logarithmic compactification of $M$.
	
	Then $\Omega_{\bar{M}}^p(\log \Div)$ is Viehweg-big if $91 \leq p \leq 101 = \dim M$.
\end{thm}

Classical examples of connected smooth complex quasiprojective varieties supporting a $\C$-PVHS are \emph{locally symmetric varieties}, i.e. the quotients of bounded symmetric domains $\Dom \subset \C^\di$ ($\di \in \N$) by the action of a torsion-free lattice. The Bergman metric of $\Dom$ induces a Kähler metric on the quotients of $\Dom$ called the \emph{canonical metric}. Zucker describes in \cite{zucker1981locally} a way to construct $\C$-PVHSs on locally symmetric varieties. We can apply the foregoing to those $\C$-PVHSs to obtain intermediate hyperbolicity properties of locally symmetric varieties.

\bigskip

In this case, we also prove a ``non-hyperbolicity'' result (note that the following statement is an equivalence). The values of $l$ are given in Remark \ref{rems:end}.\ref{item:tab}.

\begin{thm}\label{thm:lsavbig}
	Let $\Lsav$ be the quotient of a bounded symmetric domain by the action of a torsion-free lattice. Denote by $\Ricu$ the Riemann curvature tensor of the canonical metric of $\Lsav$. Let $(\bar{\Lsav}, \Div)$ be a smooth projective logarithmic compactification of $\Lsav$. Let $1 \leq p \leq \dim \Lsav$.
	
	Then the vector bundle $\Omega_{\bar{\Lsav}}^p(\log \Div)$ is L-big. Moreover, the following are equivalent:
	\begin{itemize}
		\item $\Omega_{\bar{\Lsav}}^p(\log \Div)$ is Viehweg-big,
		\item $\Omega_{\bar{\Lsav}}^p(\log \Div)$ is ample with respect to $\Lsav$,
		\item $p > l$ with $l$ the maximal dimension of $\ker \Ricu(v,\bar{v})$ for $v \in T_{\Lsav}$ with $v \neq 0$.
	\end{itemize}
\end{thm}

\begin{exs}
	\begin{itemize}
		\item If $\Lsav$ is the quotient of a complex ball, then $l = 0$.
		\item If $\Lsav$ is the quotient of a polydisc of dimension $\di \in \Ns$, then $l = \di - 1$.
		\item If $\Lsav$ is the moduli space of principally polarized abelian varieties of dimension $g \in \Ns$ with level $n \geq 3$ structure, then $l = \frac{g(g-1)}{2}$ (in this case $\dim \Lsav = \frac{g(g+1)}{2}$).
	\end{itemize}
\end{exs}

When $\Lsav$ is compact, the ampleness of $\Omega_{\Lsav}^p$ when $p > l$ was proven by Noguchi and Sunada (see \cite{noguchi1982finiteness}), without using variational Hodge theory. See Remark \ref{rems:end}.\ref{item:nopvhs} about our use of variational Hodge theory in the general case.

\subsection{Augmented base locus}

Consider $(\Comp, \Div)$ a smooth projective logarithmic compactification of a smooth quasiprojective variety $\Qpv$. Consider the augmented base locus $\AugBL(\Tautp)$ of the tautological line bundle $\Tautp$ on $\Proj \Omega_{\Comp}^p(\log \Div)$ (we use Grothendieck's convention: the set of points of $\Proj \Omega_{\Comp}^p(\log \Div)$ above a point $x \in \Comp$ is the set of vector hyperplanes of $\Omega_{\Comp}^p(\log \Div)_x$, i.e. the set of vector lines of $\Lambda^p T_{\Comp}(-\log \Div)_x$). If $\Qpv$ is connected, this augmented base locus controls the positivity of $\Omega_{\Comp}^p(\log \Div)$ (see \cite{bauer2015positivity}):
\begin{itemize}
	\item $\Omega_{\Comp}^p(\log \Div)$ is L-big if and only if $\AugBL(\Tautp) \neq \Proj \Omega_{\Comp}^p(\log \Div)$,
	\item $\Omega_{\Comp}^p(\log \Div)$ is ample with respect to a Zariski open subset $U \subset \Comp$ if and only if the projection of $\AugBL(\Tautp)$ on $\Comp$ does not intersect $U$,
	\item $\Omega_{\Comp}^p(\log \Div)$ is Viehweg-big if and only if the projection of $\AugBL(\Tautp)$ on $\Comp$ is not $\Comp$.
\end{itemize}
We prove that when $\Qpv$ supports a $\C$-PVHS $\V$ with quasi-unipotent local monodromy, $\AugBL(\Tautp)$ is controlled by the Higgs field and the conjugate Higgs field.
\begin{nota}\label{nota:c}
	Let $\Anman$ be a complex manifold, $x \in \Anman$, $\V$ a $\C$-PVHS on $\Anman$. Let $\theta_x: T_{\Anman,x} \to \End_\C(\V_x)$ be the Higgs field and $\theta_x^*: T_{\Anman,x}^{0,1} \to \End_\C(\V_x)$ the conjugate Higgs field at $x$. For $v \in T_{\Anman,x}$, we denote $C_{\V}(v) := \{w \in T_{\Anman,x}, [\theta_x^*(\bar{v}), \theta_x(w)] = 0\}$.
\end{nota}

\begin{thm}\label{thm:main}
	Let $\Comp$ be a smooth complex projective variety. Let $\Div \subset \Comp$ be a normal crossing divisor. Let $\V$ be a $\C$-PVHS on $\Qpv := \Comp \setminus \Div$ with quasi-unipotent local monodromy. Let $x \in \Qpv$ be a point where the period map of $\V$ is immersive. Let $p$ be an integer such that $1 \leq p \leq \dim \Qpv$. Let $\alpha \in \Lambda^p T_{\Qpv,x}$ be such that $[\alpha] \in \AugBL(\Tautp)$. Then there exists $v \in T_{\Qpv,x} \setminus \{0\}$ such that $\alpha \in \Lambda^p C_{\V}(v)$ (\Cref{nota:c}).
\end{thm}

We give another description of the augmented base locus that does not involve $\theta^*$ in \Cref{subsubsection:augblrk}.

\bigskip

We prove more precise results for locally symmetric varieties. First, we define higher degree characteristic subvarieties, that generalize the characteristic subvariety $\Charac_{r-1}$ defined by Mok in the case $p = 1$ and $\Dom$ irreducible in \cite[Appendix III]{mok1989metric}.

\begin{defi}
	Let $\Lsav$ be the quotient of a bounded symmetric domain by the action of a torsion-free lattice. Let $\Ricu$ be the Riemann curvature tensor of its canonical metric, let $1 \leq p \leq \dim \Lsav$. The \emph{higher degree characteristic subvariety} $\Charac^p \subset \Proj \Omega_{\Lsav}^p$ is the set of $[\alpha] \in \Proj \Omega_{\Lsav}^p$ such that there exists $v \in T_{\Lsav}$ with $v \neq 0$ such that $\alpha \in \Lambda^p \ker \Ricu(v,\bar{v})$.
\end{defi}

These higher degree characteristic subvarieties can also be defined using the Lie algebra associated with the biholomorphism group of $\Dom$, see \Cref{defi:hcsv}. They are closely linked with $\C$-PVHS because of the following proposition.

\begin{prop}\label{prop:lsavcharac}
	Let $\Lsav$ be the quotient of a bounded symmetric domain by the action of a torsion-free lattice.
	 
	Then there exists a finite étale cover $\Lsav'$ of $\Lsav$, a smooth projective logarithmic compactification of $\Lsav'$ and a $\C$-PVHS $\V$ on $\Lsav'$ with immersive period map and unipotent local monodromy and such that for every $1 \leq p \leq \dim \Lsav$, $\Charac^p \subset \Proj \Omega_{\Lsav'}^p$ is the set of $[\alpha] \in \Proj \Omega_{\Lsav'}^p$ such that there exists $x \in \Lsav'$ and $v \in T_{\Lsav',x} \setminus \{0\}$ such that $\alpha \in \Lambda^p C_{\V}(v)$ (\Cref{nota:c}).
\end{prop}

\begin{rem}
	If moreover $\Dom$ is irreducible and has rank at least two, then the previous statement holds for every non-unitary $\C$-PVHS on $\Lsav$ (see \Cref{prop:lsavcharac2}).
\end{rem}

Thus \Cref{thm:main} gives the inclusion $\AugBL(\Tautp) \cap \Proj \Omega_{\Lsav}^p \subset \Charac^p$. We prove that it is actually an equality.

\begin{thm}\label{thm:lsav}
	Let $\Lsav$ be the quotient of a bounded symmetric domain by the action of a torsion-free lattice. Let $(\bar{\Lsav}, \Div)$ be a smooth projective logarithmic compactification of $\Lsav$. Let $1 \leq p \leq \dim \Lsav$. Let $\Tautp$ be the tautological line bundle on $\Proj \Omega_{\bar{\Lsav}}^p(\log \Div)$.
	
	Then $\AugBL(\Tautp) \cap \Proj \Omega_{\Lsav}^p = \Charac^p$.
\end{thm}

This theorem generalizes to higher degree the equality $\AugBL(\Stru_1(1)) \cap \Proj \Omega_{\Lsav}^1 = \Charac^1$ that Mok proved in the case where $\Dom$ is irreducible and $\Lsav$ is compact in \cite[Appendix IV]{mok1989metric} and that was proven in the case where $\Dom$ is irreducible and of rank at least $2$ in \cite[Theorem 1.2]{rapion2025isotrivialityfamiliescurvesparametrized}.

\subsection{Outline of the proofs}

We begin by the proofs of \Cref{thm:lbig,thm:main}, which are based on the study of the positivity of the curvature of the Hodge metric which comes from a polarization of the $\C$-PVHS. Indeed, a polarization of $\V$ induces a polarization of $\End_{\C}(\V)$, that gives a Hodge metric on $\End_{\C}(\V)$. Outside of the Zariski closed subset $\Deg(\V) \subset \Qpv$ of points where the period map is not immersive, we can pull back this Hodge metric by the Higgs field $\theta$ to obtain a Kähler metric $h$ on $\Qpv \setminus \Deg(\V)$. It induces a Hermitian metric on $\Lambda^p T_{\Qpv \setminus \Deg(\V)}$, and then a Hermitian metric $\hat{h}$ on $\Tautp$. The negativity properties of the Hodge metric imply positivity properties of $\hat{h}$. In particular, this allows to prove \Cref{thm:lbig}, in \Cref{subsection:proof1}. We also obtain the following proposition, that is proven in \Cref{subsection:curvhodge}.
\begin{prop}\label{prop:curvhodge}
	Let $\Anman$ be a complex manifold. Let $\V$ be a complex polarized variation of Hodge structure on $\Anman$ with an immersive period map. Let $p$ be an integer such that $1 \leq p \leq \dim \Anman$. Let $C_1(\Tautp, \hat{h})$ be the first Chern form of $\hat{h}$ (notation given above). Let $x \in \Anman$, let $\alpha \in \Lambda^p T_{\Anman,x} \setminus \{0\}$.
	
	Then the real $(1,1)$-form $C_1(\Tautp, \hat{h})$ is semipositive. Moreover, if it is not positive definite at $[\alpha]$ then there exists $v \in T_{\Anman,x} \setminus \{0\}$ such that $\alpha \in \Lambda^p C_{\V}(v)$ (\Cref{nota:c}).
\end{prop}
We can deduce information on the augmented base locus from it thanks to Tosatti's analytic version of Nakamaye's theorem (\cite[Theorem 4.5]{tosatti2018}, see \Cref{thm:tosnak}). The latter theorem has a nefness hypothesis. We do not prove that $\Tautp$ is nef, but we compare it with the dual of a subquotient of a log-Higgs bundle, that is nef (see \Cref{thm:nef}). Then we obtain the proof of \Cref{thm:main} (\Cref{subsection:augpvhs}).

\bigskip

The key result to deduce \Cref{thm:vbig} from \Cref{thm:main} is that a tangent vector $v \in T_{\Qpv,x}$ induces a decomposition of the complex Hodge structure $\V_x$ compatible with the endomorphisms $\theta_x(w)$ for every $w \in C_{\V}(v)$, where $\theta$ is the Higgs field (\Cref{prop:decomp}). The proof of this decomposition is based on Schur’s lemma, applied in a semisimple category of ``complex polarizable infinitesimal variations of Hodge structure'' ($\C$-PIVHS, see \Cref{defi:pivhs}). We deduce \Cref{thm:vbig} from \Cref{prop:decomp} in \Cref{subsection:conseqdec}.

\bigskip

In the last section, we treat the case of locally symmetric varieties. We recall in \Cref{subsection:zucker} how to construct $\C$-PVHSs on a locally symmetric variety $\Lsav$. In \Cref{subsection:pvhslsav}, we use this construction to prove in \Cref{prop:existspvhs} the existence of a real polarized variation of Hodge structure with good properties on $\Lsav$ (we will use it in the proof of \Cref{prop:lsavcharac} and \Cref{thm:lsav}). In \Cref{subsection:higher}, we explain that if the universal cover $\Dom$ of $\Lsav$ is irreducible and of rank at least $2$, then thanks to Margulis's superrigidity theorem, every $\C$-PVHS on $\Lsav$ comes from the construction recalled in \Cref{subsection:zucker}, up to a finite étale cover (see \Cref{prop:classification}).

\bigskip

In \Cref{subsection:hcsv}, we define the higher degree characteristic subvarieties $\Charac^p$ (see \Cref{defi:hcsv}). We prove that they are Zariski closed, see \Cref{prop:charac}. The key tool used to prove the latter proposition is Peterzil and Starchenko's o-minimal Chow lemma \cite[Corollary 4.5]{peterzil2009complex}. To apply it, we have to prove that the $\Charac^p$'s are definable in an o-minimal structure. It is proven as a consequence of a theorem of Klingler, Ullmo and Yafaev \cite[Theorem 1.9]{klingler2016hyperbolic}. In \Cref{subsection:lsavaug}, we prove \Cref{prop:lsavcharac} as an application of \Cref{prop:existspvhs}, and we prove \Cref{thm:lsav} by applying Nakamaye's theorem \cite[Theorem 1.1]{nakamaye2000stable} (see \Cref{thm:cas}). Knowing that the $\Charac^p$'s are Zariski closed is crucial for our application of this theorem. Finally, \Cref{thm:lsavbig} is an immediate corollary of \Cref{thm:lsav}.

\begin{ack}
	I would like to thank my PhD advisor, Yohan Brunebarbe, for the fruitful discussions, many suggestions and constant support throughout the preparation of this paper. Thanks to Benoît Cadorel and Ariyan Javanpeykar for their careful reading and their feedback as referees of my thesis manuscript.
\end{ack}

\section{Augmented base locus and positivity}

In this section, we recall the definition and the properties of the augmented base locus, and its link with the positivity of vector bundles. Then we give a method to compute it and to prove bigness in the case of the line bundles $\Tautp$, using metric information. We give the proof of \Cref{thm:lbig}.

\subsection{Augmented base locus}\label{subsection:augdef}

The augmented base locus was first defined in \cite{ein2006asymptotic}, in continuation of \cite{nakamaye2000stable}. Let $\LB$ be a line bundle over a smooth complex projective variety $\Prv$. Denote by $\StBL(\LB)$ its stable base locus, that is the set of points of $\Prv$ on which every global section of $L^m$ vanishes for every $m \in \Ns$. One can easily check that $\bigcap_{m \in \Ns} \StBL(\LB^m \otimes \Amp^{-1})$ does not depend on the choice of an ample line bundle $\Amp$ on $\Prv$. So we can define:

\begin{defi}\label{defi:augbl}
	The \emph{augmented base locus} of $\LB$ is:
	\[
	\AugBL(\LB) := \bigcap_{m \in \Ns} \StBL(\LB^m \otimes \Amp^{-1})
	\]
	where $\Amp$ is any ample line bundle.
\end{defi}

The augmented base locus of a nef line bundle can be computed thanks to Nakamaye's following theorem.

\begin{thm}[see {\cite[Theorem 1.1]{nakamaye2000stable}}]\label{thm:cas}
	If $\LB$ is nef, then $\AugBL(\LB)$ is the union of the irreducible closed subsets $Z \subset \Prv$ of positive dimension such that the intersection number $\LB^{\dim Z} \cdot Z$ is $0$.
\end{thm}

The following lemma is easy to deduce from the definition of $\AugBL$.

\begin{lem}\label{lem:twolb}
	Let $\LB$ and $\LB'$ be two line bundles on a complex projective variety $\Prv$. Let $f: \LB \to \LB'$ be a morphism of $\Stru_{\Prv}$-modules. Let $U \subset \Prv$ be an open subset in restriction to which $f$ is an isomorphism. Then $\AugBL(\LB') \cap U \subset \AugBL(\LB)$.
\end{lem}

We will also need the following proposition of Boucksom, Broustet and Pacienza.

\begin{prop}[see {\cite[Proposition 2.3]{boucksom2013uniruledness}}]\label{prop:birat}
	Let $\Prv$ and $\Prv'$ be two smooth complex projective varieties. Let $f: \Prv \to \Prv'$ be a birational morphism, let $C \subset \Prv$ be its exceptional locus. Let $\LB$ be a line bundle on $\Prv'$. Then $\AugBL(f^*\LB) = f^{-1}(\AugBL(\LB)) \cup C$.
\end{prop}

In this article, we will consider the augmented base locus of the tautological line bundle $\Tautp$ on $\Proj \Omega_{\Comp}^p(\log \Div)$ with $(\Comp, \Div)$ a smooth projective logarithmic compactification of a quasi-projective variety $\Qpv$ and $1 \leq p \leq \dim \Qpv$. For $p=1$, \cite[Proposition 2.20]{rapion2025isotrivialityfamiliescurvesparametrized} states that this augmented base locus is compatible with finite étale covering maps. The exact same proof works for higher degrees, hence we have the following proposition.

\begin{prop}\label{prop:fecover}
	Let $f: \Qpv \to \Qpv'$ be a finite étale covering map between smooth complex quasi-projective varieties. Let $f_{\Proj}: \Proj \Omega_{\Qpv}^p \to \Proj \Omega_{\Qpv'}^p$ be the morphism induced by $f$. Let $(\Comp, \Div)$ (resp. $(\Comp', \Div')$) be a smooth projective logarithmic compactification of $\Qpv$ (resp. $\Qpv'$). Let $\Tautp$ (resp. $\Tautpp$) be the tautological line bundle of $\Proj \Omega_{\Comp}^p(\log \Div)$ (resp. $\Proj \Omega_{\Comp'}^p(\log \Div')$). Then $\AugBL(\Tautp) \cap \Proj \Omega_{\Qpv}^p = f_{\Proj}^{-1}(\AugBL(\Tautpp) \cap \Proj \Omega_{\Qpv'}^p)$.
\end{prop}

\subsection{Positivity of vector bundles}\label{subsection:augpos}

Let $\VB$ be a vector bundle on an irreducible complex projective variety $\Prv$. Denote by $\pi$ the projection $\Proj \VB \to \Prv$. The augmented base locus $\AugBL(\Taut)$ of the tautological line bundle $\Taut$ on $\Proj \VB$ is linked with positivity properties of $\VB$. See \cite{bauer2015positivity} for the proof of the following facts.

\bigskip

We have $\AugBL(\Taut) \neq \Proj \VB$ if and only if $\VB$ is \emph{L-big}, which means that $\Taut$ is big, or equivalently that there exists $\epsilon \in \R_{>0}$ such that for every $k \gg 1$, we have $\dim H^0(\Prv,\Sym^k\VB) \geq \epsilon k^{\dim(\Prv)+\rk(\VB)-1}$.

\bigskip

Let $U \subset \Prv$ be an open subset. We have $\pi(\AugBL(\Taut)) \cap U = \emptyset$ if and only if $\VB$ is \emph{ample with respect to} $U$. That means that there exists an ample line bundle $\Amp$ on $\Prv$, two integers $a, b \in \Ns$ and a morphism $\Amp^{\oplus a} \to \Sym^b\VB$ that is surjective over $U$. The vector bundle $\VB$ is said to be \emph{Viehweg-big} (with reference to \cite{viehweg1983weak}) if $\pi(\AugBL(\Taut)) \neq \Prv$, that is equivalent to the existence of a non-empty open subset with respect to which $\VB$ is ample. Viehweg-bigness is stronger than L-bigness, and for line bundles, both notions coincide with usual bigness.

\bigskip

Finally, $\VB$ is ample if and only if $\AugBL(\Taut) = \emptyset$. Brotbek and Deng proved that if $\Div$ is a non-empty normal crossing divisor in a connected smooth complex projective variety $\Comp$, then $\Omega_{\Comp}^1(\log \Div)$ is not ample (see \cite[Section 2.3]{brotbek2018positivity}). The maximal possible positivity in this case is ``almost ampleness'' (see \cite[Definition 2.1]{brotbek2018positivity}), that is stronger than ampleness with respect to $\Comp \setminus \Div$.

\subsection{Metric tools}\label{subsection:metrtools}

The bigness of a line bundle can be deduced from metric information thanks to Boucksom's following theorem. See \cite[Definition 2.1 \& 2.2]{brunebarbe2020hyperbolicity} for the definition of singular Hermitian metrics of semipositive curvature.

\begin{thm}[see {\cite[Theorem 1.2]{boucksom2002volume}}]\label{thm:boucbig}
	Let $\Anman$ be a connected compact Kähler manifold. Let $\LB$ be a line bundle on $\Anman$. Suppose that there exists a nonempty open subset $U \subset \Anman$ and a singular Hermitian metric of semipositive curvature on $\LB$ such that the restriction of its curvature to $U$ is a Kähler form (i.e. is smooth and positive definite). Then $\LB$ is big.
\end{thm}

Information on the augmented base locus of a line bundle can be deduced from metric information in the following way. If $\Anman$ is a compact complex manifold, denote by $H_{\partial\bar{\partial}}^{1,1}(\Anman,\R)$ its real Bott-Chern cohomology group (if $\Anman$ is Kähler, this group is $H^{1,1}(\Anman) \cap H^2(\Anman, \R)$).

\begin{defi}[see {\cite[Definition 3.16]{boucksom2004divisorial}}]
	Let $\Anman$ be a compact complex manifold. The non-Kähler locus of a big class $c \in H_{\partial\bar{\partial}}^{1,1}(\Anman,\R)$ is the (closed) set of points of $\Anman$ at which the Lelong number of every Kähler current in the class $c$ is positive.
\end{defi}

In case of a non big class $c \in H_{\partial\bar{\partial}}^{1,1}(\Anman,\R)$, we will call non-Kähler locus of $c$ the whole manifold $\Anman$. The non-Kähler locus is interesting for us as it allows one to compute augmented base loci using singular metrics of semipositive curvature thanks to Boucksom's following theorem.

\begin{thm}[see {\cite[Corollaire 2.2.8]{boucksom2002cones}} or {\cite[Theorem 2.3]{tosatti2018}}]\label{thm:augnkl}
	Let $\Prv$ be a smooth complex projective variety. Let $\LB$ be a line bundle on $\Prv$. Then the non-Kähler locus of $c_1(\LB)$ is $\AugBL(\LB)$.
\end{thm}

If a class is big and nef, then we can compute its non-Kähler locus using positive currents instead of Kähler currents thanks to Tosatti's following theorem.

\begin{thm}[see {\cite[Theorem 4.5]{tosatti2018}}]\label{thm:tosnak}
	Let $\Anman$ be a compact complex manifold. Let $c$ be a nef and big real $(1,1)$-class on $\Anman$. Let $T$ be a positive closed current in the class $c$. Let $U \subset \Anman$ be an open subset on which the restriction of $T$ is a Kähler form. Then the intersection of $U$ and of the non-Kähler locus of $c$ is empty.
\end{thm}

We consider now the case of a vector bundle. Let $\VB$ be a holomorphic vector bundle of rank $r \in \Ns$ on a complex manifold $\Anman$ of dimension $\di \in \N$. Let $h$ be a (smooth) Hermitian metric on $\VB$. Let $R$ be the curvature tensor of the Chern connection of $h$. Let $x \in \Anman$ and $e \in \VB_x$ be such that $h(e,e) = 1$. Then there exists an open neighborhood $U$ of $x$ in $\Anman$ and holomorphic sections $e_1, \dots, e_r$ of $\VB$ on $U$ such that $e_r(x) = e$, $h_x(e_k,e_l) = \delta_{kl}$ and $\dif h_x(e_k,e_l) = 0$ for every $1 \leq k, l \leq r$, with $\delta$ the Kronecker symbol (see \cite[Lemma 1 p.26]{mok1989metric}). Denote by $(e_1^*, \dots, e_r^*)$ the dual basis of $(e_1, \dots, e_r)$ (each $e_k^*$ is a holomorphic section of the dual $\VB^\vee$ of $\VB$). For $1 \leq k < r$, denote by $u_k$ the quotient $\frac{e_k^*}{e_r^*}$. Up to shrinking $U$, it admits holomorphic coordinates $z_1, \dots, z_{\di}$. Denote by $v_1, \dots, v_{\di}$ the associated basis of holomorphic vector fields.

\bigskip

Consider $\Proj \VB^\vee$ and the projection $\pi: \Proj \VB^\vee \to \Anman$. Then:
\[(\pi^*z_1, \dots, \pi^*z_{\di}, u_1, \dots, u_{r-1})\]
defines holomorphic coordinates in a neighborhood of $[e]$ in $\Proj \VB^\vee$. Let $\hat{h}$ be the Hermitian metric induced by $h$ on the tautological line bundle $\Taut$ on $\Proj \VB^\vee$. Denote by $C_1(\Taut, \hat{h})$ its first Chern form. Then at $[e] \in \Proj \VB^\vee$ we have the following equality (see \cite[Proposition 1 p.38]{mok1989metric}):
\begin{equation}\label{eq:cherntaut}
	C_1(\Taut, \hat{h})_{[e]} = \frac{1}{2 \pi} \sum_{k = 1}^{r-1} \dif u_k \wedge \dif \bar{u}_k - \frac{1}{2\pi} \sum_{1 \leq i, j \leq \di} h_{[e]}(\Ricu(v_i, \bar{v}_j).e, e) \dif z_i \wedge \dif \bar{z}_j.
\end{equation}

\begin{lem}\label{lem:calcnkl}
	Let $\Anman$ be a complex manifold. Let $\VB$ be a holomorphic vector bundle on $\Anman$. Let $h$ be a Hermitian metric on $\VB$. Let $R$ be the curvature tensor of the Chern connection of $h$, suppose it is seminegative in the sense of Griffiths. Let $1 \leq p \leq \rk \VB$. Let $\hat{h}$ be the Hermitian metric induced by $h$ on the tautological line bundle $\Tautp$ on $\Proj \Lambda^p \VB^\vee$. Let $C_1(\Tautp, \hat{h})$ be the first Chern form of $\hat{h}$. Let $x \in \Anman$, let $\alpha \in \Lambda^p \VB_{x} \setminus \{0\}$.
	
	Then the real $(1,1)$-form $C_1(\Tautp, \hat{h})$ is semipositive. Moreover, it is not positive definite at $[\alpha]$ if and only if there exists $v \in T_{\Anman,x} \setminus \{0\}$ such that $\alpha \in \Lambda^p K(v)$ with $K(v)$ the kernel of the Hermitian endomorphism $\Ricu(v, \bar{v})$ of $\VB_x$.
\end{lem}

\begin{proof}
	Denote by $R_p$ the curvature tensor of the Chern connection of $\Lambda^p \VB$ with the Hermitian metric induced by $h$. For $v \in T_{\Anman,x}$ and $e_1, \dots, e_p \in \VB_{x}$, we have:
	\[ R_p(v,\bar{v}).(e_1 \wedge \dots \wedge e_p) = \sum_{i=1}^p e_1 \wedge \dots \wedge e_{i-1} \wedge R(v,\bar{v}).e_i \wedge e_{i+1} \wedge \dots \wedge e_p. \]
	As $R(v,\bar{v})$ is seminegative Hermitian, $R_p(v,\bar{v})$ is seminegative Hermitian too. Hence formula (\ref{eq:cherntaut}) gives that $C_1(\Tautp, \hat{h})$ is semipositive, and that it is not positive definite at $[\alpha]$ if and only if there exists $v \in T_{\Anman,x} \setminus \{0\}$ such that $\alpha \in \ker R_p(v,\bar{v})$. Finally, we have $\ker R_p(v,\bar{v}) = \Lambda^p \ker R(v,\bar{v})$ as it holds for every semipositive or seminegative Hermitian endomorphism by the spectral theorem.
\end{proof}

The following proposition shows how to use \Cref{lem:calcnkl} to compute non-Kähler loci. We will not directly use this proposition, as the hypotheses of compactness and nefness are too restrictive for our cases. However its proof gives the outline of the proof of \Cref{thm:main}.

\begin{prop}\label{prop:resnkl}
	Let $\Anman$ be a compact Kähler manifold. Suppose the curvature $\Ricu$ is seminegative in the sense of Griffiths. Let $1 \leq p \leq \dim \Anman$. Suppose the tautological line bundle $\Tautp$ on $\Proj \Omega_{\Anman}^p$ is nef. Let $x \in \Anman$ and $\alpha \in \Lambda^p T_{\Anman,x} \setminus \{0\}$, suppose $[\alpha]$ is in the non-Kähler locus of the first Chern class $c_1(\Tautp)$. Then there exists $v \in T_{\Anman,x} \setminus \{0\}$ such that $\alpha \in \Lambda^p K(v)$ with $K(v)$ the kernel of the Hermitian endomorphism $\Ricu(v, \bar{v})$ of $T_{\Anman,x}$.
\end{prop}

\begin{proof}
	Denote by $U$ the complement of the set of $[\alpha] \in \Proj \Omega_{\Anman}^p$ such that there exists $v \in T_{\Anman,x} \setminus \{0\}$ such that $\alpha \in \Lambda^p K(v)$. By \Cref{lem:calcnkl}, $U$ is open and the semipositive real $(1,1)$-form $C_1(\Tautp, \hat{h})$ is Kähler on $U$ (with $h$ the Kähler metric of $\Anman$ and $\hat{h}$ the metric induced on $\Tautp$).
	
	Suppose $U$ is not empty (otherwise there is nothing to prove). Then we have $\int_{\Proj \Omega_{\Anman}^p} C_1(\Tautp, \hat{h})^{\dim \Proj \Omega_{\Anman}^p} > 0$. Hence \Cref{thm:boucbig} gives that $\Tautp$ is a big line bundle.
	
	As the real $(1,1)$-class $c_1(\Tautp)$ is big and, by hypothesis, nef, \Cref{thm:tosnak} applied with the closed positive real $(1,1)$-from $C_1(\Tautp, \hat{h})$ gives that the non-Kähler locus of $c_1(\Tautp)$ does not intersect $U$.
\end{proof}

\begin{rem}\label{rem:symmetry}
	Let $\Anman$ be a Kähler manifold, suppose the curvature $\Ricu$ is seminegative in the sense of Griffiths. Let $x \in \Anman$ and $v, w \in T_{\Anman,x}$. Then $w \in K(v)$ if and only if $v \in K(w)$. It is obvious if $v$ and $w$ are linearly dependent. Otherwise, as $\Ricu(v, \bar{v})$ is seminegative, it vanishes on $w$ if and only if $\left < \Ricu(v, \bar{v})(w), w \right > = 0$, which means that the holomorphic bisectional curvature in the direction of the plane generated by $v$ and $w$ is zero.
\end{rem}

\subsection{Proof of the L-bigness}\label{subsection:proof1}

With \Cref{lem:calcnkl}, we can prove \Cref{thm:lbig}. The proof is close to the proof of \cite{brunebarbe2020hyperbolicity} for the cases $p = 1$ and $p = \dim \Qpv$. Recall that a $\C$-PVHS on a complex manifold $\Anman$ induces a Hermitian metric $h$ of seminegative curvature on the open subset $U$ of points of $\Anman$ at which the period map is immersive. The proof is based on the following fact due to Brunebarbe, Klingler and Totaro.

\begin{lem}[see {\cite[Lemma 1.4]{brunebarbe2013symmetric}}]\label{lem:brklto}
	For every $x \in U$, there exists $v \in T_{\Qpv,x} \setminus \{0\}$ such that $\ker \Ricu(v,\bar{v}) = \{0\}$ with $\Ricu$ the Riemann curvature tensor of $h$.
\end{lem}

\begin{proof}[Proof of \Cref{thm:lbig}]
	By \cite[Proposition 2.4]{brunebarbe2020hyperbolicity}, $h$ extends as a singular Hermitian metric of seminegative curvature on $\Omega_{\Comp}^1(\log \Div)$. Hence it induces a singular Hermitian metric of seminegative curvature on $\Omega_{\Comp}^p(\log \Div)$, and then on $\Tautp$ (it is sufficient to check it over $U$ thanks to \cite[Lemma 2.3]{brunebarbe2020hyperbolicity}).

	Let $x$ be any point of $U$. Let $v$ be a vector as given by \Cref{lem:brklto} just above. Let $v_2, \dots, v_p$ be tangent vectors at $x$ such that $v, v_2, \dots, v_p$ are linearly independent. By \Cref{rem:symmetry}, for every $w \in T_{U,x} \setminus \{0\}$, we have $v \notin \ker \Ricu(w,\bar{w})$. Thus, $\alpha := v \wedge v_2 \wedge \dots \wedge v_p \notin \Lambda^p \ker \Ricu(w, \bar{w})$.
	
	Thus by \Cref{lem:calcnkl}, the curvature of the singular Hermitian metric of $\Tautp$ is Kähler on a Euclidean open neighborhood of $[\alpha]$ in $\Proj \Omega_{\Comp}^p(\log \Div)$. We conclude by applying \Cref{thm:boucbig}.
\end{proof}

\section{Varieties supporting a complex PVHS}

\subsection{Curvature of the Hodge metric}\label{subsection:curvhodge}

In this section, we prove \Cref{prop:curvhodge}.

\bigskip

As we will consider only \emph{complex} Hodge structures, it is not a loss of generality to consider that the weight is always $0$, and then to consider simple (and not double) gradings. Thus, a \emph{complex Hodge structure} ($\C$-HS) is a finite dimensional complex vector space $H$ with a grading $H = \bigoplus_{i \in \Z} H^i$. A \emph{polarization} $Q$ of this $\C$-HS is a conjugate-symmetric sesquilinear form on $H$ with respect to which the decomposition $\bigoplus_{i \in \Z} H^i$ is orthogonal and such that for every $i \in \Z$, the restriction of $(-1)^iQ$ to $H^i$ is positive definite. The \emph{Hodge metric} associated with $Q$ is the unique Hermitian product of $H$ with respect to which the decomposition $\bigoplus_{i \in \Z} H^i$ is orthogonal such that its restriction to $H^i$ is equal to $(-1)^iQ$ for every $i \in \Z$.

\bigskip

Let $\Anman$ be a complex manifold. Denote by $C^{\infty}$ the sheaf of rings of smooth complex functions on $\Anman$, by $A^1$ the $C^{\infty}$-module of smooth differential $1$-forms and by $A^{0,1}$ (resp. $A^{1,0}$) the submodule of $1$-forms of bidegree $(0,1)$ (resp. $(1,0)$). A \emph{complex variation of Hodge structure} ($\C$-VHS) on $\Anman$ is the data of a complex local system $\V$ on $\Anman$ and for every $x \in \Anman$, of a $\C$-HS $\bigoplus_{i \in \Z} \V_x^i$ on $\V_x$ such that for every $i \in \Z$:
\begin{itemize}
	\item the $\V_x^i$'s define a $C^{\infty}$-submodule $\V^i$ of $\V_{C^{\infty}} := \V \otimes_{\C} C^{\infty}$,
	\item if $\nabla$ is the flat connection induced by $\V$ on $\V_{C^{\infty}}$ and if $\sigma$ is a local section of $\V^i$, then the local section $\nabla \sigma$ of $\V \otimes_{\C} A^1$ is a section of $(\V^{i-1} \otimes_{C^{\infty}} A^{1,0}) \oplus (\V^i \otimes_{C^{\infty}} A^1) \oplus (\V^{i+1} \otimes_{C^{\infty}} A^{0,1})$.
\end{itemize}
A \emph{polarization} $Q$ of $\V$ is a conjugate-symmetric sesquilinear form $\V \otimes_{\C} \bar{\V} \to \C$ such that for every $x \in \Anman$, $Q_x$ is a polarization of the $\C$-HS $\V_x$. The \emph{Hodge metric} associated with $Q$ is the Hermitian metric defined on $\V_{C^{\infty}}$ by the Hodge metrics associated with the $Q_x$'s. A \emph{complex polarizable variation of Hodge structure} ($\C$-PVHS) is a $\C$-VHS that admits a polarization.

\bigskip

For $\V$ a $\C$-VHS, $i \in \Z$ and $\sigma$ a local section of $\V^i$, we have a decomposition $\nabla \sigma = \theta \sigma + \delta' \sigma + \delta'' \sigma + \theta^* \sigma$ where $\theta \sigma$ (resp. $\delta' \sigma$, $\delta'' \sigma$, $\theta^* \sigma$) is a section of $\V^{i-1} \otimes_{C^{\infty}} A^{1,0}$ (resp. $\V^i \otimes_{C^{\infty}} A^{1,0}$, $\V^i \otimes_{C^{\infty}} A^{0,1}$, $\V^{i+1} \otimes_{C^{\infty}} A^{0,1}$). A direct computation shows that $\theta$ (resp. $\theta^*$) is defined by a section of $\End_{\C}(\V) \otimes_{\C} A^{1,0}$ (resp. $\End_{\C}(\V) \otimes_{\C} A^{0,1}$). We call this section the \emph{Higgs field} (resp. \emph{conjugate Higgs field}). If $\V$ is polarizable, then $\theta$ and $\theta^*$ are adjoint with respect to every polarization.

\bigskip

The \emph{Higgs bundle} associated with a $\C$-VHS $\V$ is the data $(E,\theta)$ of the unique holomorphic vector bundle $E$ whose sheaf of smooth sections is $\V_{C^{\infty}}$ and whose Dolbeault operator is $\delta''$ defined above, and of the holomorphic global section of $\End_{\Stru_{\Anman}}(E) \otimes_{\Stru_{\Anman}} \Omega_{\Anman}^1$ defined by the Higgs field $\theta$. The vector bundle $E$ is equipped with a holomorphic grading, and for every tangent vector $v$, $\theta(v)$ is homogeneous of degree $-1$. If $\V$ is polarizable, then the Hodge metric associated with a polarization is a Hermitian metric on $E$.

\bigskip

In the following, $\Anman$ is a complex manifold, $\V$ a complex polarized variation of Hodge structure on $\Anman$, $(\VB, \theta)$ the associated Higgs bundle and $\theta^*$ the conjugate Higgs field.

\bigskip

The curvature of the Chern connection of the Hodge metric on $\VB$ is $- \theta \wedge \theta^* - \theta^* \wedge \theta$ (see \cite[Definition 4.2.8]{sabbah2019mhm}). The local system $\End_{\C}(\V)$ has an induced structure of $\C$-PVHS and an induced polarization. A direct computation shows that its Higgs field is given by:
\begin{equation}\label{eq:higgs}
	\Theta(v).f = [\theta(v), f]
\end{equation}
for $x \in \Anman$, $v \in T_{\Anman, x}$ and $f \in \End_{\C}(\VB_x)$. Hence the curvature tensor $\Ricu_{\VB}$ of $\End_{\Stru_{\Anman}}(\VB)$ over $\Anman$ is given for $x \in \Anman$, $v, w \in T_{\Anman, x}$ and $f \in \End_{\C}(\VB_x)$ by:
\begin{equation}\label{eq:curv}
	\Ricu_{\VB}(v,\bar{w}).f = - [[\theta(v), \theta^*(\bar{w})], f].
\end{equation}

\begin{lem}\label{lem:crocurv}
	Let $x \in \Anman$, $v \in T_{\Anman, x}$ and $f \in \End_{\C}(\VB_x)^{-1}$. Suppose $[\theta(v), f] = 0$. Then $[\theta^*(\bar{v}), f] = 0$ if and only if $\Ricu_{\VB}(v, \bar{v}).f = 0$.
\end{lem}

\begin{proof}
	The direct implication follows from the formula of the curvature (\ref{eq:curv}) and the Jacobi formula.
	
	For the converse implication, let $\sl \subset \End_{\C}(\VB_x)$ be the complex Lie subalgebra of endomorphisms of trace zero. Consider its real form $\su$ of endomorphisms of trace zero which are skew-adjoint with respect to the polarization. Denote by $f^*$ the conjugate of $f \in \sl$ for this real form. This is consistent with the notation $\theta^*$ in the sense that $(\theta(v))^* = \theta^*(\bar{v})$. Denote by $\sl_0$ (resp. $\su_0$) the intersection of $\sl$ (resp. $\su$) with $\End_{\C}(\VB_x)^0$. Note that $\su_0$ is a real form of $\sl_0$ and is compactly embedded in $\su$ (which means that $\su_0$ is the Lie algebra associated with a compact Lie subgroup of the real Lie group of inner automorphisms of $\su$). Let $\beta$ be the Killing form of $\su$. By \cite[Proposition 6.8]{helgason2001differential}, $\beta$ is negative definite on $\su_0$.
	
	Now suppose $\Ricu_{\VB}(v, \bar{v}).f = 0$. As $[\theta^*(\bar{v}), f] \in \sl_0$, by the foregoing, it is sufficient to prove $\beta([\theta^*(\bar{v}), f], [\theta(v), f^*]) = 0$. We have:
	\begin{align*}
		\beta([\theta^*(\bar{v}), f], [\theta(v), f^*]) &= \beta([[\theta^*(\bar{v}), f], \theta(v)], f^*)\\
		&= -\beta([[f, \theta(v)], \theta^*(\bar{v})], f^*) -\beta([[\theta(v), \theta^*(\bar{v})], f], f^*)\\
		&= -\beta([[\theta(v), \theta^*(\bar{v})], f], f^*)\\
		&= \beta(\Ricu_{\VB}(v, \bar{v}).f, f^*)\\
		&= 0.
	\end{align*}
\end{proof}

Now, suppose that the period map is immersive: equivalently, for every point $x \in \Anman$, the linear map $T_{\Anman,x} \to \End_{\C}(\VB_x)$ given by $\theta$ is injective. Consider the Hermitian metric on $T_{\Anman}$ defined as the pull back by $\theta$ of the Hodge metric on $\End_{\C}(\V)$. By \cite[Theorem 1.2]{lu1999geometry}, it is a Kähler metric. Denote by $\Ricu$ its Riemann curvature tensor.

\begin{lem}\label{lem:knc}
	Let $x \in \Anman$, $v, w \in T_{\Anman, x}$. If $\Ricu(v, \bar{v}).w = 0$, then $[\theta^*(\bar{v}), \theta(w)] = 0$
\end{lem}

\begin{proof}
	The Hodge metric on $\End_{\Stru_{\Anman}}(\VB)$ has Griffiths seminegative curvature. As the injection $\theta: T_{\Anman} \hookrightarrow \End_{\Stru_{\Anman}}(\VB)$ is locally split and the curvature decreases on direct factor subbundles, we have $\theta(\ker \Ricu(v,\bar{v})) \subset \ker \Ricu_{\VB}(v, \bar{v})$. Moreover, $\theta(v)$ and $\theta(w)$ commute. Hence we can apply \Cref{lem:crocurv}.
\end{proof}

\begin{proof}[Proof of \Cref{prop:curvhodge}]
	Apply \Cref{lem:calcnkl} to the tangent bundle and apply \Cref{lem:knc}.
\end{proof}

\subsection{Augmented base locus and Hodge metric}\label{subsection:augpvhs}

In this section, we prove \Cref{thm:main}, using the same method as for \Cref{prop:resnkl}. Let $\Qpv$ be a smooth complex quasiprojective variety. Let $(\Comp, \Div)$ be a smooth projective logarithmic compactification of $\Qpv$. Let $\V$ be a $\C$-PVHS on $\Qpv$. Suppose its local monodromy is unipotent. This hypothesis is not restrictive to prove \Cref{thm:main}. Indeed, for every $\C$-PVHS $\V$ with quasi-unipotent local monodromy on $\Qpv$, there exists a finite étale covering map $\Qpv' \to \Qpv$ such that the pull-back of $\V$ on $\Qpv'$ has unipotent local monodromy (see \cite[Proposition 8.2]{bakker2024linear}), and the augmented base locus is compatible with finite étale covering maps (\Cref{prop:fecover}).

\bigskip

Let $(\VB, \theta)$ be the Higgs bundle associated with $\V$. Denote by $\theta^*$ the conjugate Higgs field. Denote by $\Deg(\V)$ the \emph{degeneracy locus} of $\V$, that is the set of points $x \in \Qpv$ such that the linear map $T_{\Qpv,x} \to \End_{\C}(\VB_x)$ defined by $\theta$ is not injective. Equivalently, it is the set of points at which the period map is not immersive. It is a Zariski closed subset of $\Qpv$.

\bigskip

The following lemma is essentially \cite[Example 7.17.3]{hartshorne2013algebraic}.

\begin{lem}\label{lem:sublb}
	Let $k$ be a field. Let $\Qpv$ be an algebraic $k$-variety. Let $\VB$ be a locally free coherent $\Stru_{\Qpv}$-module. Let $\LB$ be an invertible coherent submodule of $\VB$. Let $U \subset \Qpv$ be the biggest open subset of $\Qpv$ on which $\LB$ is locally a direct factor of $\VB$. Then there exist a blow-up $b: \Qpv' \to \Qpv$ whose center does not intersect $U$ and an invertible coherent submodule $\LB'$ of $b^*\VB$ that is locally a direct factor, includes $b^*\LB$ and is equal to it above $U$.
\end{lem}

\begin{proof}
	The injection of $\LB$ in $\VB$ induces a morphism $a: \VB^\vee \otimes \LB \to \Stru_{\Qpv}$. The image $\I$ of $a$ is a coherent sheaf of ideals whose restriction to $U$ is equal to the structural sheaf. Let $b: \Qpv' \to \Qpv$ be the blow-up with respect to $\I$. Let $\I' := b^{-1}\I \cdot \Stru_{\Qpv'}$ be the inverse image ideal sheaf of $ \I$. Then $\I'$ is the image of the morphism $b^*a: b^*\VB^\vee \otimes b^*\LB \to \Stru_{\Qpv'}$, and it is an invertible coherent $\Stru_{\Qpv'}$-module. Then define $\LB'$ as $\I'^{-1}.b^*\LB$.
\end{proof}

Schmid's nilpotent orbit theorem (see \cite{schmid1973variation}) gives a canonical extension of $\VB$ as a holomorphic vector bundle $\ext{\VB}$ on $\Comp$. The Higgs field defines a morphism of holomorphic vector bundles $\ext{\theta}: T_{\Comp}(-\log \Div) \to \End_{\Stru_{\Comp}}(\ext{\VB})$ called the \emph{log-Higgs field}. It induces a morphism $\ext{\theta}_p: \Lambda^p T_{\Comp}(-\log \Div) \to \Lambda^p \End_{\Stru_{\Comp}}(\ext{\VB})$.

\bigskip

We denote by $\pi$ the projection $\Proj \Omega_{\Comp}^p(\log \Div) \to \Comp$ and by $\Stru_p(-1)$ the dual of $\Tautp$. It is locally a direct factor of $\pi^*\Lambda^p T_{\Comp}(-\log \Div)$.

\begin{lem}\label{lem:lb}
	There exist a smooth projective variety $P$, a projective birational morphism $b: P \to \Proj \Omega_{\Comp}^p(\log \Div)$ which is an isomorphism above $\Qpv \setminus \Deg(\V)$ and such that $b^{-1}(\pi^{-1}(\Div \cup \Deg(\V)))$ is a normal crossing divisor, and an invertible coherent submodule $\LB$ of $b^*\pi^*\Lambda^p\End_{\Stru_{\Comp}}(\ext{\VB})$ that is locally a direct factor, includes $b^*\ext{\theta}_p(\Stru_p(-1))$ and is equal to it above $\Qpv \setminus \Deg(\V)$.
\end{lem}

\begin{proof}
	Apply \Cref{lem:sublb} and define $P$ as a desingularization of $\Qpv'$.
\end{proof}

The following lemma shows that above $\Qpv \setminus \Deg(\V)$, the augmented base locus of $\LB^\vee$ constrains the augmented base locus of $\Tautp$.

\begin{lem}\label{lem:lbtaut}
	With the notations of the previous lemma, we have $\AugBL(\Tautp) \cap \Proj \Omega_{\Qpv \setminus \Deg(\V)}^p \subset b(\AugBL(\LB^\vee))$.
\end{lem}

\begin{proof}
	Because of \Cref{prop:birat}, we have:
	\[\AugBL(\Tautp) \cap \Proj \Omega_{\Qpv \setminus \Deg(\V)}^p \subset b(\AugBL(b^*\Tautp)).\]
	Then \Cref{lem:twolb} gives $\AugBL(b^*\Tautp) \cap b^{-1}(\Proj \Omega_{\Qpv \setminus \Deg(\V)}^p) \subset \AugBL(\LB^\vee)$.
\end{proof}

To conclude the proof of \Cref{thm:main}, we need to prove that $\LB^\vee$ has properties close to the ones needed for $\Tautp$ in the hypotheses of \Cref{prop:resnkl}. This is given by the two theorems below, that we can apply to $\LB$ because of the following lemma. The vector bundle $b^*\pi^*\Lambda^p\End_{\Stru_{\Comp}}(\ext{\VB})$ is the log-Higgs bundle associated with the $\C$-PVHS $b^*\pi^*\Lambda^p\End_{\C}(\V)$.

\begin{lem}\label{lem:vanish}
	The log-Higgs field of $b^*\pi^*\Lambda^p\End_{\Stru_{\Comp}}(\ext{\VB})$ vanishes on the line bundle $\LB$ defined in \Cref{lem:lb}.
\end{lem}

\begin{proof}
	Let $x \in \Qpv$, let $v \in T_{\Qpv,x}$. As $\theta(v)$ commutes with $\theta(w)$ for every $w \in T_{\Qpv,x}$, the Higgs field $\Theta$ of $\End_{\Stru_{\Qpv}}(\VB)$ vanishes at $\theta(v)$ because of formula (\ref{eq:higgs}). Hence the Higgs field $\Theta^p$ of $\Lambda^p \End_{\Stru_{\Qpv}}(\VB)$ vanishes on the image of $\Lambda^p T_{\Qpv}$ in $\Lambda^p \End_{\Stru_{\Qpv}}(\VB)$ (defined by $\theta$ and Leibniz rule). Therefore, the Higgs field of $\pi^* \Lambda^p \End_{\Stru_{\Qpv}}(\VB)$ vanishes on the image of $\Stru_p(-1)$. By density of $b^{-1}(\Proj \Omega_{\Qpv}^p)$ in $P$, it is equivalent to the vanishing of the log-Higgs field of $b^*\pi^*\Lambda^p\End_{\Stru_{\Comp}}(\ext{\VB})$ on $\LB$.
\end{proof}

The following two theorems due to Zuo and Brunebarbe allow us to use $\LB^\vee$ instead of $\Tautp$ to mimic the proof of \Cref{prop:resnkl}.

\begin{thm}[see {\cite[Theorem 1.2]{zuo2000negativity}} or {\cite[Theorem 1.8]{brunebarbe2018symmetric}}]\label{thm:nef}
	Let $\LB$ be an invertible coherent submodule of $\ext{\VB}$ which is locally a direct factor. Suppose the log-Higgs field vanishes on $\LB$. Then the dual $\LB^\vee$ of $\LB$ is nef.
\end{thm}

\begin{thm}[see {\cite[Theorem 1.4]{brunebarbe2017semipositivityhiggsbundles}}]\label{thm:metsing}
	Let $\LB$ be an invertible coherent submodule of $\ext{\VB}$. Suppose the log-Higgs field vanishes on $\LB$. Then the restriction of the Hodge metric of $\VB$ to $\LB$ is a singular Hermitian metric of seminegative curvature.
\end{thm}

\begin{proof}[Proof of \Cref{thm:main}]
	Denote by $U$ the complement of the set of $[\alpha] \in \Proj \Omega_{\Qpv \setminus \Deg(\V)}^p$ such that there exists $v \in T_{\Qpv \setminus \Deg(\V)}$ with $v \neq 0$ such that $\alpha \in \Lambda^p C(v)$. We have to prove that $U$ does not intersect $\AugBL(\Tautp)$. By \Cref{lem:lbtaut}, it is sufficient to prove that $b^{-1}(U)$ does not intersect $\AugBL(\LB^\vee)$ with $b$ and $\LB$ defined in \Cref{lem:lb}.
	
	By \Cref{thm:metsing} applied to $\LB$ thanks to \Cref{lem:vanish}, the Hodge metric induces a singular metric $\hat{h}$ of semipositive curvature on $\LB^\vee$. It means that $C_1(\LB^\vee, \hat{h})$ defines a positive closed real $(1,1)$-current on $P$ whose de Rham class is the first Chern form $c_1(\LB^\vee)$. By \Cref{prop:curvhodge} applied on $\Qpv \setminus \Deg(\V)$, and because $\LB^\vee$ and $\Tautp$ are equal above it, $b^{-1}(U)$ is included in a Euclidean open subset $U' \subset P$ on which $C_1(\LB^\vee, \hat{h})$ is Kähler.
	
	Suppose $U'$ is not empty (otherwise there is nothing to prove). Then we have $\int_P C_1(\LB^\vee, \hat{h})^{\dim P} > 0$. Hence \Cref{thm:boucbig} gives that $\LB^\vee$ is a big line bundle.
	
	The real $(1,1)$-class $c_1(\LB^\vee)$ is big, and nef by \Cref{thm:nef} (applied to $\LB$ thanks to \Cref{lem:vanish}). Hence we can apply \Cref{thm:tosnak} with the positive closed real $(1,1)$-form $C_1(\LB^\vee, \hat{h})$. It gives that the non-Kähler locus of $c_1(\LB^\vee)$ does not intersect $U'$. Thanks to \Cref{thm:augnkl}, the non-Kähler locus of $c_1(\LB^\vee)$ is the augmented base locus of $\LB^\vee$.
\end{proof}

\subsection{Decomposition of polarized infinitesimal variations of Hodge structure}\label{subsection:decomp}

To deduce the positivity properties of the logarithmic cotangent bundles from \Cref{thm:main}, we will need \Cref{prop:decomp} below. To state and prove it, we need to define a notion of complex polarizable infinitesimal variation of Hodge structure. 

\begin{defi}\label{defi:pivhs}
	Let $T$ be a complex vector space. A \emph{complex polarizable infinitesimal variation of Hodge structure} ($\C$-PIVHS) on $T$ is the data of $(H, H^i, \delta, \delta^*)$ with $(H, H^i)$ a $\C$-HS, $\delta: T \to \End_{\C}(H, H^i)^{-1}$ a linear map and $\delta^*: T \to \End_{\C}(H, H^i)^1$ an antilinear map such that:
	\begin{itemize}
		\item for every $v, w \in T$, $\delta(v)$ and $\delta(w)$ commute,
		\item there exists a polarization $Q$ of $(H, H^i)$ with respect to which for every $v \in T$, $\delta(v)$ and $\delta^*(v)$ are adjoint.
	\end{itemize}
	
	A \emph{morphism of $\C$-PIVHS} $(H, H^i, \delta, \delta^*) \to (H', H'^i, \delta', \delta'^*)$ on $T$ is a morphism $f: H \to H'$ of $\C$-HS which is equivariant for $\delta$ and $\delta'$, and for $\delta^*$ and $\delta'^*$ (i.e. for every $v \in T$ and $a \in H$, we have $f(\delta(v).a) = \delta'(v).f(a)$ and $f(\delta^*(v).a) = \delta'^*(v).f(a)$).
\end{defi}

\begin{exs}
	\begin{itemize}
		\item Consider $\Anman$ a complex manifold, $x \in \Anman$, and $\V$ a $\C$-PVHS on $\Anman$. Denote by $\V_x$ the fiber of $\V$ above $x$, $\V_x^i$ its $\C$-HS, $\theta_x$ the Higgs field and $\theta_x^*$ the conjugate Higgs field at $x$. Then $(\V_x, \V_x^i, \theta_x, \theta_x^*)$ is a $\C$-PIVHS on $T_{\Anman,x}$.
		\item If $(H_\Z, H^{i,j}, Q, T, \delta)$ is an infinitesimal variation of Hodge structure of weight $n \in \Z$ in the sense of Carlson-Green-Griffiths-Harris \cite{carlson1983infinitesimal}, then $(H_\Z \otimes_\Z \C, H^{i,n-i}, \delta, \bar{\delta})$ is a $\C$-PIVHS on $T$, where $\bar{\delta}$ is the conjugate of $\delta$ for the natural real form of $H_\Z \otimes_\Z \C$.
	\end{itemize}
\end{exs}

\begin{rem}
	The definition is made so that we can apply Schur's lemma. In particular:
	\begin{itemize}
		\item We do not use the ``integral'' infinitesimal variations of Hodge structure defined by Carlson, Green, Griffiths and Harris. Indeed, even if we are only interested by integral, rational or real variations of Hodge structure, it is necessary to have a $\C$-linear category to apply Schur's lemma.
		\item We do not use the ``naive'' complex analog of integral infinitesimal variations of Hodge structure, that would be the data of $(H, H^i, \delta)$ without $\delta^*$. Indeed, the category of ``naive'' $\C$-PIVHSs on a complex vector space $T$ is not semi-simple. Consider for instance a line $T = \C v$, a $\C$-HS $H = H^0 \oplus H^1$ with $H^0$ and $H^1$ two lines, and $\delta(v)$ an isomorphism $H^1 \xrightarrow{\sim} H^0$. Then $H^0$ is stable by $\delta(v)$ but has no complement in $H$ stable by $\delta(v)$.
	\end{itemize}
\end{rem}

\begin{prop}\label{prop:semis}
	Let $T$ be a complex vector space. The category $\Cat_T$ of $\C$-PIVHSs on $T$ is semi-simple Abelian.
\end{prop}

\begin{proof}
	It is easy to see that $\Cat_T$ is additive and admits kernels. Let $H$ be an object of $\Cat_T$ and $H' \subset H$ a sub-object, i.e. a sub-$\C$-HS stable by $\delta(v)$ and $\delta^*(v)$ for every $v \in T$. Consider the orthogonal $H'^{\bot}$ for a polarization of $H$ as in \Cref{defi:pivhs}. It is a sub-$\C$-HS stable by $\delta(v)$ (resp. $\delta^*(v)$) because $H'$ is stable by its adjoint $\delta^*(v)$ (resp. $\delta(v)$), for every $v$. Hence we have $H = H' \oplus H'^{\bot}$ in the category $\Cat_T$. This proves semi-simplicity and the existence of quotients.
\end{proof}

Let $H$ and $H'$ be two objects of $\Cat_T$. We can define $H \otimes H'$ as an object of $\Cat_T$ by letting $v \in T$ act by $\delta(v) \otimes \id_{H'} + \id_H \otimes \delta(v)$ and $\delta^*(v) \otimes \id_{H'} + \id_H \otimes \delta^*(v)$. If $T'$ is a complex vector space, then a linear map $f: T \to T'$ induces an exact functor $f^*: \Cat_{T'} \to \Cat_T$. Hence, using the pullbacks by the projections, for $S$ a complex vector space, an object $H$ of $\Cat_S$ and an object $H'$ of $\Cat_T$ define an object $H \otimes H'$ of $\Cat_{S \oplus T}$.

\begin{prop}\label{prop:decomp}
	Let $S, T$ be two complex vector spaces. Let $(H, H^i, \delta, \delta^*)$ be a $\C$-PIVHS on $S \oplus T$. Then the following are equivalent.
	\begin{itemize}
		\item For every $v \in S$ and $w \in T$, $\delta^*(v)$ and $\delta(w)$ commute.
		\item There exists a finite set $A$ and a family $(V_a)$ (resp $(W_a)$) of $\C$-PIVHS on $S$ (resp. $T$) indexed by $A$ such that $H$ is isomorphic to $\bigoplus_{a \in A} V_a \otimes W_a$ (in the category $\Cat_{S \oplus T}$).
	\end{itemize}
\end{prop}

\begin{proof}
	We prove the direct implication, the converse is obvious. Denote $H_S := i_S^*H$ with $i_S: S \hookrightarrow S \oplus T$ the inclusion and $H_T := i_T^*H$ with $i_T: T \hookrightarrow S \oplus T$ the inclusion. The group $\Z$ acts freely by shift of the grading on the set of isomorphism classes of simple objects of $\Cat_S$. Let $A$ be the finite set of orbits under this action which are represented by a simple sub-$\C$-PIVHS of $H_S$. For every $a \in A$, choose a $\C$-PIVHS $V_a$ on $S$ which represents $a$.
	
	Let $W_a$ be the vector space of linear maps $V_a \to H$ which are equivariant for $\delta_{\vert S}$ and $\delta_{\vert S}^*$ (but that are not necessarily morphisms of $\C$-HS). Then $W_a$ is naturally a $\C$-HS, so it can be considered as an object of $\Cat_S$ with zero $\delta$ and $\delta^*$.
	
	Consider the following morphism of $\C$-PIVHS on $S$.
	\[\application{\alpha}{\bigoplus_{a \in A} V_a \otimes W_a}{H_S}{v_a \otimes w_a}{w_a(v_a).}\]
	Schur's lemma gives the injectivity of $\alpha$, and \Cref{prop:semis} gives the surjectivity.
	
	For every $w \in T$, $\delta(w)$ is an element of $\End_\C(H)^{-1}$. By definition of a $\C$-PIVHS, $\delta(w)$ is equivariant for $\delta_{\vert S}$. By hypothesis, it is equivariant for $\delta_{\vert S}^*$. Hence, by Schur's lemma, using the isomorphism $\alpha$, $\delta(w)$ can be considered as an element of $\bigoplus_{a \in A} \End_\C(W_a)^{-1}$. The same applies for $\delta^*(w)$, which can be considered as an element of $\bigoplus_{a \in A} \End_\C(W_a)^1$. Denote by $\delta_{W_a}$ and $\delta_{W_a}^*$ the morphisms $T \to \End_\C(W_a)$ defined in this way for every $a \in A$.
	
	It only remains to check that for every $a \in A$, the structure $(W_a, W_a^i, \delta_{W_a}, \delta_{W_a}^*)$ that we just defined is a $\C$-PIVHS on $T$, by proving that it admits a convenient polarization. A polarization $Q$ of $\C$-HS of $H$ and a polarization $Q'$ of $\C$-HS of $V_a$ induce a polarization $Q''$ of $\C$-HS of $W_a$, as it is a sub-$\C$-HS of $\Hom_\C(V_a, H)$. By definition of $\C$-PIVHS, we can choose $Q$ such that for every $w \in T$, $\delta(w)$ and $\delta^*(w)$ are adjoint. Moreover, one can easily check that for every $w \in T$ and $f \in W_a$, we have $\delta_{W_a}(w).f = \delta(w) \circ f$ and $\delta_{W_a}^*(w).f = \delta^*(w) \circ f$, so $\delta_{W_a}(w)$ and $\delta_{W_a}^*(w)$ are adjoint with respect to $Q''$.
\end{proof}

\subsection{Rank criterion}\label{subsection:conseqdec}

Let $(H, H^i, \delta, \delta^*)$ be a $\C$-PIVHS on a complex vector space $T$. For $v \in T, i \geq j \in \Z$, we will denote by $v^{i,j}$ the restriction $H^i \to H^j$ of the iterate $\delta(v)^{i-j}$ and by $\delta^{i,j}$ the map $v \mapsto v^{i,j}$.

\subsubsection{Decomposed PIVHS and rank}

The existence of the decomposition given by \Cref{prop:decomp} allows one to prove properties of the rank of $\theta_x^{i,j}(w)$ for $w \in C_{\V}(v)$ and for some integers $i$ and $j$, in the context of \Cref{thm:main}.

\begin{lem}\label{lem:rk}
	Let $(H, H^i, \delta, \delta^*)$ be a $\C$-PIVHS on a complex vector space $T$. Let $S \subset T$ be a vector subspace. Let $v \in T$ be such that $\delta(v) \neq 0$ and for every $w \in S$, $[\delta^*(v),\delta(w)] = 0$. Let $i \in \Z$ be maximal such that $v^{i,i-1} \neq 0$. Then there exists $j \in \Z$, $\lambda \in \C \setminus \{0\}$ and $u \in S$ such that $j < i$ and for every $w \in S$, $\rk((\lambda v + u)^{i,j}) > \rk(w^{i,j})$.
\end{lem}

\begin{proof}
	Applying \Cref{prop:decomp} to the pull-back of $H$ on $\C v \oplus S$, we get a decomposition $H = \bigoplus_{a \in A} X_a$ (with $A$ a finite set) and for every $a \in A$, $X_a = V_a \otimes W_a$ with $V_a$ a $\C$-PIVHS on $\C v$ and $W_a$ a $\C$-PIVHS on $S$. We can refine the decomposition to suppose that every $V_a$ is simple. For every $a \in A$ and $j \in \Z$ such that $j < i$, the set of $(\lambda, u) \in \C \times S$ for which $\rk((\lambda v + u)_{\vert X_a}^{i,j})$ is maximal is a non-empty Zariski open subset of $\C \times S$. Hence, the intersection of these subsets is a non-empty Zariski open subset of $\C \times S$. In particular, there exists $(\lambda, u) \in (\C \setminus \{0\}) \times S$ such that for every $a \in A$, every $j \in \Z$ such that $j < i$ and every $w \in S$, we have $\rk((\lambda v + u)_{\vert X_a}^{i,j}) \geq \rk(w_{\vert X_a}^{i,j})$. Now, it is sufficient to prove that there exists $a$ and $j$ such that $(\lambda v + u)_{\vert X_a}^{i,j} \neq 0$ and for every $w \in S$, we have $w_{\vert X_a}^{i,j} = 0$.
	
	By definition of $i$, there exists $a \in A$ such that the restriction of $v^{i,i-1}$ to $X := X_a$ is non-zero. From now on, we denote $V := V_a$, $W := W_a$ and the endomorphisms of $H$ defined by $\delta$ are considered in restriction to $X$. We denote by $v_V$ the action of $v$ on $V$ and by $w_W$ the action of an element $w \in S$ on $W$.
	
	Let $s \in \Z$ be maximal such that $W^s \neq 0$, let $t \in \Z$ with $t \leq s$ be minimal such that there exists $w \in S$ such that $w_W^{s,t} \neq 0$. Let $k \in \Z$ be maximal such that $V^k \neq 0$. By simplicity of $V$, $k$ is also maximal such that $v_V^{k,k-1} \neq 0$. Let $l \in \Z$ with $l < k$ be minimal such that $v_V^{k,l} \neq 0$. Then we have $i = k + s$. Let $j := l + t$: because $l < k$ and by definition of $t$, we have $w^{i,j} = 0$ for every $w \in S$. It remains to check that $(\lambda v + u)^{i,j} \neq 0$:
	\begin{align*}
	(\lambda v + u)^{i,j} &= (\lambda v_V \otimes \id_W + \id_V \otimes u_W)^{i,j}\\
	&= \sum_{l'+t'=j} \binom{i-j}{k-l'} \lambda^{k-l'} v_V^{k,l'} \otimes u_W^{s,t'} \\
	&= \binom{i-j}{k-l} \lambda^{k-l} v_V^{k,l} \otimes u_W^{s,t}\\
	&\neq 0.
	\end{align*}
\end{proof}

The following two propositions are simplified forms of the previous lemma.

\begin{prop}\label{prop:rkpivhs}
	Let $(H, H^i, \delta, \delta^*)$ be a $\C$-PIVHS on a complex vector space $T$. Let $S \subset T$ be a vector subspace. Suppose there exists $v \in T$ such that $\delta(v) \neq 0$ and for every $w \in S$, $[\delta^*(v),\delta(w)] = 0$.
	
	Then there exists $i, j \in \Z$ and $z \in T$ such that $j < i$ and for every $w \in S$, $\rk(z^{i,j}) > \rk(w^{i,j})$.
\end{prop}

\begin{proof}
	Take $z := \lambda v + u$ as given by the previous lemma.
\end{proof}

\begin{prop}\label{prop:rkcy}
	In the setting of the previous proposition, let $i \in \Z$ be maximal such that $\delta^{i,i-1} \neq 0$.
	
	If $\delta^{i,i-1}$ is injective, then the conclusion of the proposition can be realized with this $i$.
	
	If moreover $\dim H^i = 1$, then the conclusion of the proposition can be realized with this $i$ and for $j$ the minimal integer such that $\delta^{i,j} \neq 0$. In this case this conclusion is equivalent to: for every $w \in S$, $w^{i,j} = 0$.
\end{prop}

\begin{proof}
	For the first part, the injectivity of $\delta^{i,i-1}$ guarantees that $v^{i,i-1} \neq 0$, so $i$ is the one in the hypothesis of \Cref{lem:rk}.
	
	For the second part, we have by the first part $j' \in \Z$ and $z' \in T$ such that $j' < i$ and for every $w \in S$, $\rk(z'^{i,j'}) > \rk(w^{i,j'})$. As $\dim H^i = 1$, this implies $w^{i,j'} = 0$. Composing with $w^{j',j}$, we have $w^{i,j} = 0$. By definition of $j$, there exists $z \in T$ such that $z^{i,j} \neq 0$. Hence $\rk(z^{i,j}) = 1 > 0 = \rk(w^{i,j})$.
\end{proof}

\subsubsection{Augmented base locus and rank}\label{subsubsection:augblrk}

Combining the study of infinitesimal variations of Hodge structures of the previous sections with \Cref{thm:main}, we obtain the following descriptions of the augmented base loci associated with $\Qpv$.

\begin{prop}\label{prop:blrk}
	Let $\Comp$ be a smooth complex projective variety. Let $\Div \subset \Comp$ be a normal crossing divisor. Let $\V$ be a $\C$-PVHS on $\Qpv := \Comp \setminus \Div$. Suppose $\V$ has quasi-unipotent local monodromy. Let $\theta$ be the associated Higgs field. Let $x \in \Qpv \setminus \Deg(\V)$. Let $p \in \N$ be such that $1 \leq p \leq \dim \Qpv$. Let $\Tautp$ be the tautological line bundle on $\Proj \Omega_{\Comp}^p(\log \Div)$. Let $[\alpha] \in \AugBL(\Tautp)$ be above $x$.
	
	Then there exists a vector subspace $S \subset T_{\Qpv, x}$, two integers $i$ and $j$ with $i > j$ and $z \in T_{\Qpv, x}$ such that $\alpha \in \Lambda^p S$ and for every $w \in S$, we have $\rk \theta^{i,j}(w) < \rk \theta^{i,j}(z)$.
\end{prop}

\begin{proof}
	Apply \Cref{thm:main}: then the conclusion results from \Cref{prop:rkpivhs} applied to the $\C$-PIVHS defined by $\V$ on $T_{\Qpv,x}$, with $C_{\V}(v)$ as $S$.
\end{proof}

\begin{prop}\label{prop:blrkcy}
	In the setting of the previous proposition, let $i \in \Z$ be the biggest integer such that $\theta_x^{i,i-1}$ is nonzero. If:
	
	$(*)$ $\theta_x^{i,i-1}$ is injective,
	
	then the conclusion of the proposition can be realized with this $i$.
	
	If moreover:
	
	$(**)$ $\dim \V_x^i = 1$,
	
	then the conclusion of the proposition can be realized with $i$ and for $j$ the minimal integer such that $\theta_x^{i,j} \neq 0$.  In this case this conclusion is equivalent to: for every $w \in S$, $\theta_x^{i,j}(w) = 0$.
\end{prop}

\begin{proof}
	Apply \Cref{prop:rkcy}.
\end{proof}

\subsubsection{Positivity and rank}\label{subsubsection:poslogcot}

Finally, we can deduce \Cref{thm:vbig} and its generalization \Cref{thm:general} from the properties of the augmented base locus proven in the previous section.

\begin{thm}\label{thm:general}
	Let $\Comp$ be a connected smooth complex projective variety. Let $\Div \subset \Comp$ be a normal crossing divisor. Let $\V$ be a $\C$-PVHS on $\Qpv := \Comp \setminus \Div$. Suppose $\V$ has quasi-unipotent local monodromy. Let $x \in \Qpv \setminus \Deg(\V)$. Let $m$ be the maximum dimension for a vector subspace $S \subset T_{\Qpv,x}$ such that there exist two integers $j < i$ and a vector $z \in T_{\Qpv,x}$ such that for every $w \in S$, $\rk \theta_x^{i,j}(w) < \rk \theta_x^{i,j}(z)$. Let $p \in \N$ be such that $m < p \leq \dim \Qpv$. Then $\Omega_{\Comp}^p(\log \Div)$ is Viehweg-big.
\end{thm}

\begin{proof}
	Combine \Cref{prop:blrk} with \Cref{subsection:augpos}.
\end{proof}

\begin{rem}
	In the statement of the previous theorem, $m$ depends on $x$. If the inequality $p > m$ is realized for every $x$ in an open subset $U \subset \Qpv \setminus \Deg(\V)$, then $\Omega_{\Comp}^p(\log \Div)$ is ample with respect to $U$.
\end{rem}

\begin{proof}[Proof of \Cref{thm:vbig}]
	Combine \Cref{prop:blrkcy} with \Cref{subsection:augpos}.
\end{proof}

\subsection{Example: the moduli space of quintic threefolds}

Let $n$ and $d$ be two positive integers. Denote by $\mathcal{M}$ the fine moduli space of smooth hypersurfaces of degree $d$ of $\Proj_{\C}^{n+1}$ polarized by the tautological line bundle $\Taut$. It is a smooth Deligne-Mumford stack of finite type over $\C$. If $U \subset \Proj H^0(\Proj_{\C}^{n+1}, \Stru(d))^\vee$ is the open subset of points $[f]$ such that $\{f= 0\}$ is smooth, then $\mathcal{M}$ is the stack quotient $[\PGL_{n+2} \backslash U]$. Let $M$ be a connected smooth complex quasiprojective variety with a finite étale covering map $M \to \mathcal{M}$. Javanpeykar and Loughran proved that if $d \geq 3$ and $(n, d) \neq (1, 3)$, then there exists $a \in \Ns$ such that for every $N \geq 3$ coprime to $a$, the moduli space of hypersurfaces with a level $N$ structure gives an example of such a $M$ (see \cite[Theorem 3.2]{javanpeykar2017moduli}). We will apply \Cref{thm:vbig} to $M$.

\bigskip

Let $f$ be a non-zero homogeneous polynomial of $\C[X_0, \dots, X_{n+1}]$ of degree $d$. Consider the hypersurface $Y := \{f = 0\}$ polarized by $\Taut_{\vert Y}$. For $p,q \in \Z$, denote by $\Hprim^{p,q}(Y)$ the primitive cohomology of $Y$ of bidegree $(p,q)$. Denote by $I := \left ( \frac{\partial f}{\partial X_0}, \dots, \frac{\partial f}{\partial X_{n+1}} \right )$ the Jacobian ideal of $f$ and by $R := \C[X_0, \dots, X_{n+1}] / I$ its Jacobian ring. As $I$ is graded, $R$ has a structure of graded ring. By a theorem of Griffiths, there are canonical isomorphisms $\Hprim^{n-q,q}(Y) \cong R_{d(q+1)-n-2}$ for every $0 \leq q \leq n$ and $T_{M,x} \cong R_d$ (see \cite[Example p.186]{carlson1983infinitesimal}).

\bigskip

Let $\Prv$ be the pull-back on $M$ of the universal family of $\mathcal{M}$. The primitive cohomology of degree $n$ of fibers of $\Prv$ define an integral polarized variation of Hodge structure $\V_{\Z}$ on $M$. Let $x \in M$ be a point. Let $f \in H^0(\Proj_{\C}^{n+1}, \Stru(d))$ be such that $\Prv_x$ and $\{f = 0\}$ are isomorphic as polarized varieties. Griffiths proved that the following square commutes:
\[
\begin{tikzcd}
T_{M,x} \otimes \Hprim^{n-q,q}(\Prv_x) \arrow{r}{\sim} \arrow{d} & R_d \otimes R_{d(q+1)-n-2} \arrow{d} \\
\Hprim^{n-q-1,q+1}(\Prv_x) \arrow{r}{\sim} & R_{d(q+2)-n-2}
\end{tikzcd}
\]
where the horizontal maps are given by the isomorphisms mentioned above, the left vertical map by the Higgs field and the right vertical map by the product in $R$.

\bigskip

Now, suppose $d = n + 2$. Then $\Hprim^{n,0} \cong R_0$ has dimension $1$. Let $x \in M$ be a point corresponding to the Fermat hypersurface $\{X_0^{n+2} + \dots + X_{n+1}^{n+2} = 0\}$. In this case we have $R = \C[X_0, \dots, X_{n+1}] / \left ( X_0^{n+1}, \dots, X_{n+1}^{n+1} \right )$. The iterate $\theta_x^{n,0}: T_{M,x} \to \Hom_{\C}(\Hprim^{n,0}(\Prv_x), \Hprim^{0,n}(\Prv_x))$ of the Higgs field is non-zero. Indeed, consider $v \in T_{M,x}$ corresponding to $X_0 \dots X_{n+1} \in R_{n+2}$ by the isomorphism $T_{M,x} \cong R_{n+2}$ mentioned above. Then we have $\theta_x^{n,0}(v) \neq 0$ because $(X_0 \dots X_{n+1})^n \neq 0$ in $R$.

\bigskip

Hence we can apply \Cref{thm:vbig} to $\V_{\Z} \otimes_{\Z} \C$ with $i = n$ and $j = 0$. Let $m$ be the biggest dimension for a vector subspace $S \subset R_{n+2}$ such that for every $g \in S$, we have $g^n = 0$. Then $\Omega_{\bar{M}}^p(\log \Div)$ is Viehweg-big for every $p \in \N$ such that $m < p \leq \dim M$, with $(\bar{M}, \Div)$ a smooth projective logarithmic compactification of $M$. To compute $m$, we use the following lemma.

\begin{lem}
	Let $Y$ be a hypersurface of degree at least $2$ in a projective space $\Proj^n$ ($n \in \Ns$). Let $c$ be the codimension in $\Proj^n$ of the singular locus of $Y$. Let $c'$ be the minimal codimension in $\Proj^n$ for a linear subspace included in $Y$. Then $c \leq 2 c'$.
\end{lem}

\begin{proof}
	See \cite[Remark 1.3.3]{huybrechts2023geometry} for the proof in the smooth case. The proof easily generalizes to give the general case.
\end{proof}

Consider the hypersurface $N := \{[g], g^n = 0\} \subset \Proj R_{n+2}^{\vee}$ (it is a hypersurface because $\dim R_{n(n+2)} = 1$). By the Jacobian criterion, the singular locus of $N$ is the set of $[g] \in N$ such that for every $h \in R_{n+2}$, we have $g^{n-1}h = 0$. By Macaulay's theorem, it is equivalent to $g^{n-1} = 0$. Hence, if $k$ denotes the dimension of $\{g, g^{n-1} = 0\} \subset R_{n+2}$, we have $m \leq \frac{k + \dim R_{n+2}}{2}$.

\bigskip

Consider the case $n = 3$ of quintic threefolds. We have $\dim M = \dim R_{n+2} = 101$. Computing in the Jacobian ring of the Fermat quintic threefold with a computer algebra system (see the SageMath code in Appendix \ref{section:appendix}), we get $k \leq 79$, hence $m \leq 90$. Thus we get \Cref{thm:quintthree}: $\Omega_{\bar{M}}^p(\log \Div)$ is Viehweg-big for every $p \in \N$ such that $91 \leq p \leq 101$.

\begin{rems}
	\begin{itemize}
		\item In the statement of \Cref{thm:quintthree}, the ``$91$'' may not be optimal. It is probable that actually $k < 79$: $79$ is only the limit of our computer algebra system and our patience.
		\item The proof shows that the statement of \Cref{thm:quintthree} holds more generally for every connected smooth complex quasiprojective variety $M$ that parametrizes a smooth polarized family $\pi: W \to M$ of quintic threefolds such that there exists a closed point $x \in M$ such that:
		\begin{itemize}
			\item the fiber $\pi^{-1}(\{x\})$ is isomorphic to the Fermat quintic threefold,
			\item the period map of the $\C$-PVHS $R^3\pi_*\underline{\C}_{W}$ is immersive at $x$.
		\end{itemize}
	\end{itemize}
\end{rems}

\section{The case of locally symmetric varieties}

In this section, we prove \Cref{prop:lsavcharac} and Theorems \ref{thm:lsav} and \ref{thm:lsavbig}.

\bigskip

Let $\di \in \N$, let $\Dom \subset \C^\di$ be a bounded symmetric domain. The biholomorphism group of $\Dom$ has a unique structure of real Lie group. We denote by $\Hol(\Dom)^+$ its identity component. It is a connected semisimple real Lie group with trivial center. Let $\Latt \subset \Hol(\Dom)^+$ be a torsion-free lattice. Denote by $\Lsav$ the quotient $\Latt \backslash \Dom$. It has a natural structure of complex manifold. In the study of $\Lsav$, we can often reduce to the case where $\Latt$ is irreducible thanks to the following lemma. Recall that when $\Latt$ is irreducible, by Margulis's arithmeticity theorem, either $\Latt$ is arithmetic or $\Dom$ is biholomorphic to a complex ball (see \cite[Theorem (1') p.4]{margulis1991discrete}).

\begin{lem}\label{lem:reduc}
	There exist a finite family $(\Dom_i)$ of bounded symmetric domains, a finite family $(\Latt_i)$ of irreducible torsion-free lattices with $\Latt_i \subset \Hol(\Dom_i)^+$ for every $i$ and a holomorphic finite étale covering map $\prod_i \Latt_i \backslash \Dom_i \to \Lsav$.
\end{lem}

\begin{proof}
	As the center of $\Hol(\Dom)^+$ is trivial, there exists a finite family of commuting normal subgroups $H_i \subset \Hol(\Dom)^+$ such that the induced morphism $\prod_i H_i \to \Hol(\Dom)^+$ is an isomorphism and the image of $\prod_i \Latt_i$ has finite index in $\Latt$, with $\Latt_i := \Latt \cap H_i$ (see \cite[Theorem 5.22]{raghunathan1972discrete}). Define $\Dom_i$ as the symmetric space associated with the real Lie group $H_i$.
\end{proof}

An application of the previous lemma is that $\Lsav$ admits a unique structure of connected smooth complex quasiprojective variety compatible with its structure of complex manifold. When $\Latt$ is irreducible, the algebraic structure is given by Baily-Borel in the case where $\Latt$ is arithmetic (see \cite[III.4]{borel2006compactifications}) and by Siu-Yau and Mok in the case where $\Dom$ is a ball (see \cite{siu1982compactification,mok2011projective}). For the general case, see Baldi-Ullmo \cite[Theorem 3.1.12]{baldi2023special}.


\subsection{Locally homogeneous variations of Hodge structure}\label{subsection:zucker}

In his interpretation of bounded symmetric domains as moduli spaces of Hodge structures, Deligne gives a way to construct variations of Hodge structure on $\Dom$ (see \cite[Corollaire 1.1.17]{deligne1979varietes}). We recall here this construction, following the more explicit description given by Zucker (\cite{zucker1981locally}, see this article for the proofs).

\bigskip

Let $\Group$ be a connected semisimple real Lie group. Let $\alpha: \Group \to \Hol(\Dom)^+$ be a surjective morphism with compact kernel. Let $A$ be a finite dimensional complex vector space and $\rho: \Group \to \GL(A)$ a representation of $\Group$. We will construct a $\C$-PVHS denoted by $\tV(\Group, A)$ on $\Dom$ starting from these data.

\bigskip

Denote by $\Group^{\ad}$ the quotient of $\Group$ by its (finite) center. As the center of $\Hol(\Dom)^+$ is trivial, $\alpha$ factorizes by a morphism $\alpha': \Group^{\ad} \to \Hol(\Dom)^+$. As $\Hol(\Dom)^+$ is semisimple of non-compact type and $\ker \alpha'$ is semisimple compact, the surjection $\alpha'$ is naturally split and $\Group^{\ad}$ is naturally isomorphic to $\Hol(\Dom)^+ \times N$ with $N := \ker(\alpha')$. In particular, the Lie algebra $\hlie$ associated with $\Hol(\Dom)^+$ naturally identifies with a direct factor of the Lie algebra $\glie$ associated with $\Group$. This gives a natural representation of $\hlie$ on $A$.

\bigskip

Denote by $\U$ the real Lie group of complex numbers of modulus $1$. The group of characters of $\U$ is freely generated by the inclusion $\chi: \U \hookrightarrow \C^*$. For every $x \in \Dom$, there exists a unique real Lie group morphism $c_x: \U \hookrightarrow \Group$ such that:
\begin{itemize}
	\item $c_x$ is injective,
	\item $p_N \circ c_x$ is trivial, with $p_{N}$ the morphism $\Group \to N$ induced by the decomposition $\Group^{\ad} \cong \Hol(\Dom)^+ \times N$,
	\item the action of $\U$ on $\Dom$ \textit{via} $\alpha \circ c_x$ fixes $x$,
	\item the weights of the action of $\U$ on the holomorphic tangent vector space $T_{\Dom,x}$ \textit{via} $\dif_x \circ \alpha \circ c_x$ are \emph{positive} powers of $\chi$.
\end{itemize}

\bigskip

Suppose $A$ is a simple representation of $\Group$. Consider the set $W$ of weights that appear in $A$ for the representation $\rho \circ c_x$ of $\U$ ($W$ does not depend on $x$). There exists $\lambda \in \Z$, $\mu \in \Ns$ and $m \in \N$ such that $W = \{\chi^{\lambda - k \mu}, 0 \leq k \leq m\}$. Denote by $\tV(\Group, A)_x^k$ the weight space of $A$ for the weight $\chi^{\lambda - k \mu}$ under the representation $\rho \circ c_x$. This gives to $A$ a $\C$-HS $\tV(\Group, A)_x$ which depends on $x \in \Dom$. Those $\C$-HSs vary smoothly with $x$. This gives a $\C$-PVHS $\tV(\Group, A)$ on $\Dom$, with the constant local system of fiber $A$ as underlying local system.

\bigskip

In general, we choose $\bigoplus_i A_i$ a decomposition of $A$ as a sum of irreducible representations. Then the $\C$-PVHS $\tV(\Group, A) := \bigoplus_i \tV(\Group, A_i)$ does not depend on the choice of the decomposition. Moreover, the action of $\Group$ on $\Dom$ extends as an action on $\tV(\Group, A)$, and $\tV(\Group, A)$ admits a polarization that is equivariant for this action.

\bigskip

Now, we give a formula to describe the Higgs field and the conjugate Higgs field of $\tV(\Group, A)$. Consider the Lie algebra $\hlie$ associated with $\Hol(\Dom)^+$ equipped with the adjoint representation $\Ad$ of $\Hol(\Dom)^+$. For every $x \in \Dom$, denote by $\llie_x$ the real Lie subalgebra of $\hlie$ which is tangent to the stabilizer $\Isot_{x} \subset \Hol(\Dom)^+$ of $x$. There exists a unique $\Isot_x$-subrepresentation $\mlie_x$ of $\hlie$ such that $\hlie = \llie_x \oplus \mlie_x$.

\bigskip

Apply the foregoing to $\hlie_\C := \hlie \otimes_{\R} \C$. We have $W = \{\chi, 1, \chi^{-1}\}$. On the one hand, we have $\tV(\Hol(\Dom)^+, \hlie_\C)_x^1 = \llie_{x,\C} := \llie_x \otimes_\R \C$. On the other hand, by differentiating the action $\Hol(\Dom)^+ \times \Dom \to \Dom$ at $(1,x)$, we get a $\R$-linear map $\hlie \to T_{\Dom,x}^\R$ (where $T_{\Dom,x}^{\R}$ is the real tangent vector space). Its kernel is $\llie_x$. Thus it induces an isomorphism $e_x: \mlie_{x,\C} \xrightarrow{\sim} T_{\Dom,x}^\C$ with $\mlie_{x,\C} := \mlie_x \otimes_{\R} \C$ and $T_{\Dom,x}^\C := T_{\Dom,x}^\R \otimes_{\R} \C$. We have $e_x(\tV(\Hol(\Dom)^+, \hlie_\C)_x^0) = T_{\Dom,x}$ and $e_x(\tV(\Hol(\Dom)^+, \hlie_\C)_x^2) = T_{\Dom,x}^{0,1}$. Denote by $\iota_x: T_{\Dom,x}^\C \xrightarrow{\sim} \mlie_{x,\C}$ the inverse of $e_x$. Then $\iota_x$ induces two isomorphisms $T_{\Dom,x} \xrightarrow{\sim} \tV(\Hol(\Dom)^+, \hlie_\C)_x^0 \subset \hlie_\C$ and $T_{\Dom,x}^{0,1} \xrightarrow{\sim} \tV(\Hol(\Dom)^+, \hlie_\C)_x^2 \subset \hlie_\C$. An easy computation gives the following lemma.

\begin{lem}\label{lem:higgslsav}
	Let $x \in \Dom$. Let $\theta_x: T_{\Dom,x} \to \End_{\C}(A)$ be the Higgs field of $\tV(\Group, A)$ at $x$ and $\theta_x^*: T_{\Dom,x}^{0,1} \to \End_{\C}(A)$ the conjugate Higgs field. Let $v \in T_{\Dom,x}$, let $a \in A$.
	
	Then $\theta_x(v).a = \iota_x(v).a$ and $\theta_x^*(\bar{v}).a = \iota_x(\bar{v}).a$.
\end{lem}

Now we recall Zucker's construction of locally homogeneous variations of Hodge structure (see \cite{zucker1981locally}). Let $\Lambda \subset \Group$ be a torsion-free lattice. The group $\Lambda$ acts by $\alpha$ on $\Dom$ and by $\rho$ on $A$. This gives an action of $\Lambda$ on $\Dom$ equipped with the constant local system of fiber $A$. The $\C$-PVHS $\tV(\Group, A)$ is equivariant under this action, and admits an equivariant polarization. Hence it descends as a $\C$-PVHS on $\Lsav = \Lambda \backslash \Dom$ that we will denote by $V(\Group, A, \Lambda)$.

\bigskip

In the case where $A$ is of the form $A_\R \otimes_\R \C$ with $A$ a simple real representation of $\Group$, the set $W = \{\chi^{\lambda - k \mu}, 0 \leq k \leq m\}$ of weights defined above is stable under inversion, i.e. $m \mu = 2\lambda$. For $0 \leq k \leq m$, let $\tV(\Group, A_\R)_x^{k,m-k} := \tV(\Group, A)_x^k$. This defines a real polarizable variation of Hodge structure ($\R$-PVHS) $\tV(\Group, A_\R)$ of weight $m$ on $\Dom$, with the constant local system of fiber $A_\R$ as underlying local system. If $A$ is of the form $A_\R \otimes_\R \C$ with $A$ a real representation of $\Group$, we define the $\R$-PVHS $\tV(\Group, A_\R)$ as the direct sum $\bigoplus_i \tV(\Group, A_{i,\R})$ with $\bigoplus_i A_{i,\R}$ a decomposition of $A_\R$ as a sum of simple representations. Again, the action of $\Group$ on $\Dom$ extends as an action on $\tV(\Group, A_\R)$, and there exists an equivariant polarization. Thus, $\tV(\Group, A_\R)$ descends as an $\R$-PVHS denoted by $V(\Group, A_\R, \Lambda)$ on $\Lambda \backslash \Dom$.

\begin{defi}
	A $\R$-PVHS $\V_\R$ on $\Lsav$ is \emph{locally homogeneous} if there exists a connected semisimple real Lie group $\Group$, a surjective morphism with compact kernel $\alpha: \Group \to \Hol(\Dom)^+$, a real representation $A_\R$ of $\Group$ and a torsion-free lattice $\Lambda \subset \Group$ such that $\alpha(\Lambda) = \Latt$ and $\V_\R \cong V(\Group, A_\R, \Lambda)$.
\end{defi}

\subsection{A useful PVHS}\label{subsection:pvhslsav}

We start by proving that up to a finite étale cover, every locally symmetric variety admits a locally homogeneous $\R$-PVHS with good properties. In particular, we will use in the proof of \Cref{prop:lsavcharac} the existence of a $\C$-PVHS with immersive period map and quasi-unipotent local monodromy, and we will use in the proof of \Cref{thm:lsav} the existence of a $\R$-PVHS with immersive period map, unipotent local monodromy and such that the log-Higgs field is compatible with some metrics and is locally split injective.

\bigskip

We denote $\tilde{\V}_\R := \tV(\Hol(\Dom)^+, \hlie)$ and $\tilde{\V} := \tilde{\V}_\R \otimes_\R \C$. Denote by $(\tilde{\VB}, \tilde{\theta})$ the Higgs bundle associated with $\tilde{\V}$.

\bigskip

Recall that the domain $\Dom$ admits a canonical Kähler metric called the \emph{Bergman metric}. The Bergman metric is equivariant under the action of $\Hol(\Dom)^+$. If $\Dom$ is irreducible, then up to a constant multiple, the Bergman metric is the only Hermitian metric equivariant under the action of $\Hol(\Dom)^+$. If $\Dom$ is biholomorphic to a product of bounded symmetric domains, then its Bergman metric is the sum of the pull-backs of the Bergman metrics of each factor.

\begin{lem}\label{lem:metric}
	There exists a polarization of $\tilde{\V}_\R$ equivariant under the action of $\Hol(\Dom)^+$ and such that the pull-back by $\tilde{\theta}$ of the Hodge metric on $\End_{\Stru_{\Dom}}(\tilde{\VB})$ is the canonical metric on $T_{\Dom}$.
\end{lem}

\begin{proof}
	Consider the decomposition in irreducible components $\Dom \cong \prod_{i} \Dom_i$, and the associated decompositions in simple factors $\Hol(\Dom)^+ = \prod_i H_i$ and $\hlie = \bigoplus_i \hlie_i$. Denote by $(\tilde{\VB}_i, \tilde{\theta}_i)$ the Higgs bundle associated with $\tV(H_i, \hlie_i)$ on $\Dom_i$. As the period map of $\tV(H_i, \hlie_i)$ on $\Dom_i$ is immersive (by \Cref{lem:higgslsav}), $\tilde{\theta}_i$ is locally split injective. Thus the pull-back by $\tilde{\theta}_i$ of the Hodge metric on $\End_{\Stru_{\Dom_i}}(\tilde{\VB}_i)$ induced by a polarization of $\tV(H_i, \hlie_i)$ is a Hermitian metric on $\Dom_i$.
	
	Let $\tilde{Q}_i$ be a polarization of $\tV(H_i, \hlie_i)$ equivariant for the action of $H_i$. Denote by $\tilde{h}_i$ the Hodge metric of $\End_{\Stru_{\Dom_i}}(\tilde{\VB}_i)$ induced by $\tilde{Q}_i$. It is easy to check that $\tilde{\theta}_i$ is equivariant for the action of $H_i$. Thus the pull-back of $\tilde{h}_i$ by $\tilde{\theta}_i$ is a Hermitian metric equivariant by the action of $H_i$. Hence we can replace $\tilde{Q}_i$ by a constant multiple to suppose that this pull-back is the Bergman metric of $\Dom_i$.

	As $\tV(\Hol(\Dom)^+, \hlie)$ is isomorphic to the sum of the pull-backs of the $\tV(H_i, \hlie_i)$, we can define a polarization $\tilde{Q}$ of $\tV(\Hol(\Dom)^+, \hlie)$ as the sum of the pull-backs of the $\tilde{Q}_i$'s. Let $\tilde{h}$ be the Hodge metric induced by $\tilde{Q}$ on $\End_{\Stru_{\Dom}}(\tilde{\VB})$. The pull-back of $\tilde{h}$ by $\tilde{\theta}$ is the Bergman metric, and $\tilde{Q}$ is equivariant under the action of $\Hol(\Dom)^+$.
\end{proof}

We denote $\V_\R := V(\Hol(\Dom)^+, \hlie, \Latt)$ and $\V = \V_\R \otimes_\R \C$. As the polarization given by the previous lemma is equivariant under the action of $\Hol(\Dom)^+$, it descends as a polarization of $\V_\R$.

\begin{lem}\label{lem:unipmonod}
	There exist a finite étale covering map $\pi: \Lsav' \to \Lsav$ and a smooth projective logarithmic compactification of $\Lsav'$ such that $\pi^*\V$ has unipotent local monodromy.
\end{lem}

\begin{proof}
	Thanks to \Cref{lem:reduc}, we can suppose $\Latt$ irreducible. We choose $\pi$ and $\Lsav'$ such that $\Lsav'$ admits a toroidal compactification $(\bar{\Lsav}', \Div)$ as constructed by Ash, Mumford, Rapoport and Tai when $\Latt$ is arithmetic (see \cite{ash2010smooth}) and by Mok when $\Dom$ is a ball (see \cite{mok2011projective}). Then, a small loop around an irreducible component of $\Div$ corresponds to a unipotent element of $\Hol(\Dom)^+$ (see Mumford \cite[p.255-256]{mumford1977hirzebruch} when $\Latt$ is arithmetic and Baldi-Ullmo \cite[Proposition 3.3.3]{baldi2023special} when $\Dom$ is a ball).
\end{proof}

As the Bergman metric is equivariant for the action of $\Latt$, it descends as Kähler metric on $\Lsav$ called the \emph{canonical metric}.

\begin{lem}\label{lem:locsplitinj}
	Consider the finite étale covering map $\pi: \Lsav' \to \Lsav$ and the smooth projective logarithmic compactification $(\bar{\Lsav}', \Div)$ of $\Lsav'$ given by \Cref{lem:unipmonod}. If $(\ext{\VB}, \ext{\theta})$ denote the log Higgs bundle defined by $\pi^*\V$ on $\bar{\Lsav}'$, then the log Higgs field $\ext{\theta}: T_{\bar{\Lsav}'}(-\log \Div) \to \End_{\Stru_{\bar{\Lsav}'}}(\ext{\VB})$ is locally split injective.
\end{lem}

\begin{proof}
	Let $\iota: \Disk^\di \hookrightarrow \bar{\Lsav}'$ be an open embedding of a polydisc such that $\iota^{-1}(\Div) \subset \{z_1 \cdot \ldots \cdot z_{\di} = 0\}$ (with $z_1, \dots, z_\di$ the coordinates of $\Disk^\di$). Let $v$ be a trivializing global section of the line bundle $\iota^* \Lambda^\di T_{\bar{\Lsav}'}(-\log \Div)$. Let $\sigma$ be the image of $v$ by $\iota^*\Lambda^\di \ext{\theta}$. It is a section of $\iota^* \Lambda^\di \End_{\Stru_{\bar{\Lsav}'}}(\ext{\VB})$. To prove the result, it is sufficient to prove that $\sigma$ does not vanish at $0$.
	
	Denote by $g^\di$ the Hermitian metric induced on $\Lambda^\di T_{\bar{\Lsav}'}(-\log \Div)$ by the canonical metric of $\Lsav'$. On the one hand, because the compactification is toroidal, then by Mumford \cite[Proposition 3.4.a)]{mumford1977hirzebruch} (arithmetic case) and Mok's description of the volume form in \cite[Theorem 1]{mok2011projective} (ball case), there exist $C \in \R_{>0}$ and $n \in \N$, such that for every $z = (z_1, \dots, z_\di) \in \Disk^\di$ in a neighborhood of $0$, we have:
	\begin{equation}\label{eq:good}
		\lVert v_z \rVert_{g^\di} \geq C \left ( \sum_{i=1}^\di - \ln \lvert z_i \rvert \right )^{-n}.
	\end{equation}
	On the other hand, $\lVert v \rVert_{g^\di}$ is equal to $\lVert \sigma \rVert_h$ with $h$ the Hodge metric of $\iota^*\Lambda^\di\End_\R(\V_\R)$ induced by the polarization of $\V_\R$ given by \Cref{lem:metric}. Suppose that $\sigma_0 = 0$. Then, as a consequence of the  Hodge metric estimates of Cattani, Kaplan and Schmid (see \cite[Theorem 5.21]{cattani1986degeneration}), we obtain that for every $a \in \R$ with $a < 1$, there exists $C' \in \R_{>0}$ such that for every $z = (z_1, \dots, z_\di) \in \Disk^\di$ in a neighborhood of $0$, we have:
	\begin{equation}\label{eq:cks}
		\lVert v_z \rVert_{g^\di} \leq C' \left ( \sum_{i=1}^\di \lvert z_i \rvert \right )^a.
	\end{equation}
	There is a contradiction between formulas (\ref{eq:good}) and (\ref{eq:cks}), thus $\sigma(0) \neq 0$.
\end{proof}

Combining all the previous lemmas, we finally get the following.

\begin{prop}\label{prop:existspvhs}
	Let $\Dom$ be a bounded symmetric domain, $\Hol(\Dom)^+$ the neutral component of its biholomorphism group, $\hlie$ the associated (semisimple real) Lie algebra, $\Latt$ a torsion-free lattice of $\Hol(\Dom)^+$ and $\Lsav$ be the quotient $\Latt \backslash \Dom$.
	
	Then there exist:
	\begin{itemize}
		\item a finite index subgroup $\Latt'$, inducing a finite étale cover $\Lsav' := \Latt' \backslash \Dom$ of $\Lsav$,
		\item a smooth projective logarithmic compactification $(\bar{\Lsav}', \Div)$ of $\Lsav'$,
		\item a polarization $Q$ of the locally homogeneous $\R$-PVHS $\V_\R := V(\Hol(\Dom)^+, \hlie, \Latt')$ on $\Lsav'$,
	\end{itemize}
	such that:
	\begin{itemize}
		\item $\V_\R$ has unipotent local monodromy (on $\bar{\Lsav}'$),
		\item if $(\VB, \theta)$ denote the Higgs bundle defined by $\V := \V_\R \otimes_\R \C$ on $\Lsav'$, then the pull-back by $\theta$ of the Hodge metric induced by $Q$ on $\End_{\Stru_{\Lsav'}}(\VB)$ is the canonical metric on $T_{\Lsav'}$,
		\item if $(\ext{\VB}, \ext{\theta})$ denote the log-Higgs bundle defined by $\V$ on $\bar{\Lsav}'$, then the map $\ext{\theta}: T_{\bar{\Lsav}'}(-\log \Div) \to \End_{\Stru_{\bar{\Lsav}'}}(\ext{\VB})$ is locally split injective.
	\end{itemize}
\end{prop}

Note moreover that the period map of $\V$ is immersive.

\begin{proof}
	We choose $\Lsav'$ and $(\bar{\Lsav}', \Div)$ as given by \Cref{lem:unipmonod}. We define $Q$ as the quotient of the polarization given by \Cref{lem:metric}. The three properties are verified thanks to \Cref{lem:unipmonod}, \Cref{lem:metric} and \Cref{lem:locsplitinj}.
\end{proof}

\subsection{The case of higher rank}\label{subsection:higher}

If we assume that $\Latt$ is irreducible and $\Dom$ has rank at least $2$, then every $\C$-PVHS comes from locally homogeneous $\C$-PVHSs in some sense (\Cref{prop:classification} below). In preparation for this statement, we define $\Group$-$\C$-HS to define more $\C$-PVHS on $\Lsav$.

\begin{defi}
	Let $\Group$ be a Lie group. A \emph{$\Group$-$\C$-HS} is a $\C$-HS equipped with an action of $\Group$ such that $\Group$ acts by automorphism of $\C$-HS.
\end{defi}

A $\C$-HS $A$ with an action of an algebraic group $\Group$ is a $\Group$-$\C$-HS if and only if for every $i \in \Z$, $A^i$ is a subrepresentation of $\Group$. Now, let $\Group$ be a connected semisimple real Lie group with a surjective morphism with compact kernel $\alpha: \Group \to \Hol(\Dom)^+$ (as in \Cref{subsection:zucker}), and let $A$ be a $\Group$-$\C$-HS. Consider the $\C$-PVHS $\tV(\Group, A^i)$ on $\Dom$, for every $i \in \Z$. Shift the grading of $\tV(\Group, A^i)$ by $i$. Sum the obtained $\C$-PVHS for every $i$: we get a $\C$-PVHS $\tV(\Group, A)$ on $\Dom$.

\bigskip

Let $\Lambda \subset \Group$ be a torsion-free lattice. Then $\tV(\Group, A)$ descends as a $\C$-PVHS $V(\Group, A, \Lambda)$ on $\Lambda \backslash \Dom$. The monodromy representation of the local system underlying $V(\Group, A, \Lambda)$ is $A$ with the action of $\Group$ restricted to $\Lambda$. If $\Dom$ has rank at least $2$ and $\Latt$ is irreducible, then up to a finite étale cover, every $\C$-PVHS on $\Lsav$ has the form $V(\Group, A, \Lambda)$ for some group $\Group$, some $\Group$-$\C$-HS $A$ and some lattice $\Lambda$.

\begin{prop}\label{prop:classification}
	Let $\Dom$ be a bounded symmetric domain of rank at least $2$. Let $\Latt \subset \Hol(\Dom)^+$ be an irreducible torsion-free lattice.

	Then there exists a connected semisimple real Lie group $\Group$, a surjective morphism with compact kernel $\alpha: \Group \to \Hol(\Dom)^+$ and a torsion-free lattice $\Lambda \subset \Group$ such that $\alpha(\Lambda)$ is a finite index subgroup of $\Latt$ and such that the following holds.
	
	Let $\V$ be a $\C$-PVHS on $\Latt \backslash \Dom$. Then there exists a $\Group$-$\C$-HS $A$ and a finite index subgroup $\Lambda' \subset \Lambda$ such that the pull-back of $\V$ on $\Lambda' \backslash \Dom$ is isomorphic to $V(\Group, A, \Lambda')$.
\end{prop}

\begin{proof}
	As $\Dom$ has rank at least $2$ and $\Latt$ is irreducible, $\Latt$ is arithmetic by Margulis's arithmeticity theorem (see \cite[Theorem (1') p.4]{margulis1991discrete}). There exist a connected $\Q$-simple $\Q$-group $\Group_\Q$, a surjective morphism of real Lie groups $\alpha: \Group_\Q(\R) \to \Hol(\Dom)^+$ with compact kernel and an arithmetic lattice (in the sense of $\Q$-groups) $\Lambda \subset \Group_\Q(\Q)$ such that $\alpha(\Lambda)$ is a finite index subgroup of $\Latt$. By the existence of algebraic universal covers for algebraic groups, we can suppose $\Group_\Q$ simply connected. By Malcev's theorem, we can suppose that $\Lambda$ is torsion-free. Denote by $\pi$ the natural finite étale covering map $\Lambda \backslash \Dom \to \Latt \backslash \Dom$.

	By \cite[Proposition 1.13]{deligne1987theoreme}, a $\C$-PVHS decomposes as a direct sum of $\C$-PVHSs with irreducible underlying local system. Hence one can suppose the local system underlying $\pi^*\V$ irreducible. We can replace $\Lambda$ by a finite index subgroup to suppose moreover that it remains irreducible after pull-back by every finite étale cover. By \cite[Proposition 1.13]{deligne1987theoreme} again, it is sufficient to prove that up to a finite étale covering map, the local system underlying $\pi^*\V$ is isomorphic to the local system underlying $V(\Group, A, \Lambda')$ for some representation $A$ of $\Group$ and some finite index subgroup $\Lambda' \subset \Lambda$. It is equivalent to prove that the representation of $\Lambda$ given by the monodromy of $\pi^*\V$ coincides with a representation of $\Group$ on a finite index subgroup $\Lambda' \subset \Lambda$. As $\Lambda$ is an arithmetic lattice of a simply connected semisimple $\Q$-group of rank at least $2$, it is given by the superrigidity theorem of Margulis for arithmetic lattices (see \cite[Theorem (6)(iii) p.5]{margulis1991discrete}).
\end{proof}

\subsection{Higher degree characteristic subvarieties}\label{subsection:hcsv}

In this section, we define the higher degree characteristic subvarieties. We prove that they are Zariski closed (that we will need to prove \Cref{thm:lsav}). For this proof, we will need the notion of Borel dual (see \cite{wolf1972fine} for the proofs). For every $x \in \Dom$, let $\Isot_{x,\C}$ be the connected complex Lie subgroup of $\GL(T_{\Dom,x})$ whose associated Lie algebra is $\llie_{x,\C}$ (the complexification of the Lie algebra $\llie_x$ tangent to the stabilizer $\Isot_x \subset \Hol(\Dom)^+$ of $x$). Let $H_\C$ be the adjoint complex Lie group of the semisimple complex Lie algebra $\hlie_\C$. The faithful action of $\Hol(\Dom)^+$ on $\hlie_\C$ by adjunction induces an injective real Lie group morphism $\Hol(\Dom)^+ \hookrightarrow H_\C$. Moreover, the restriction $\Isot_x \hookrightarrow H_\C$ extends as an injective complex Lie group morphism $i_x: \Isot_{x,\C} \hookrightarrow H_\C$. Define the \emph{Borel dual} $\Dual$ of $\Dom$ as the orbit of the morphism $i_x$ under the action of $H_\C$ by conjugation (this orbit does not depend on $x$). It admits a unique structure of complex manifold such that $H_\C$ acts biholomorphically. It is a smooth projective variety. The map $i: x \mapsto i_x$ is a holomorphic open embedding of $\Dom$ in $\Dual$, equivariant under the action of $\Hol(\Dom)^+$. Denote by $P_x$ the stabilizer of $i_x$ under the action of $H_\C$. Then the image of $P_x$ in $\Aut(T_{\Dom,x})$ is equal to $\Isot_{x,\C}$.

\bigskip

We will also need some o-minimal geometry (see \cite{van1998tame}). In this theory, some subsets of $\Qpv(\R)$ of any real algebraic variety $\Qpv$ are said \emph{semi-algebraic}, and some subsets are said \emph{definable in $\Ranexp$}. Those two classes of subsets are stable by finite unions, finite intersections, Cartesian products, complement, direct and inverse images by a morphism of real algebraic varieties, and they contain the ``trivial'' subset $\Qpv(\R)$ of $\Qpv(\R)$ for every real algebraic variety $\Qpv$. The class of semi-algebraic subsets can be defined as the smallest class that verifies these properties. The class of subsets definable in $\Ranexp$ is the smallest class of subsets that verifies the previous properties and that contains the following subsets:
\begin{itemize}
	\item the graph of the real exponential (as a subset of $\mathbb{A}^2(\R)$),
	\item the subset $\{(x, f(x)) \vert x \in [0,1]^{n}\}$  of $\mathbb{A}^{n+1}(\R)$ for every $n \in \Ns$ and every real analytic function $f$ defined in a Euclidean open neighborhood of $[0,1]^{n}$.
\end{itemize}
A map between semi-algebraic (resp. definable in $\Ranexp$) subsets is said to be \emph{semi-algebraic} (resp. \emph{definable in $\Ranexp$}) if its graph is. For every complex algebraic variety $\Qpv$, the set of closed points of $\Qpv$ canonically identifies with the set of real points of its real Weil restriction of scalar. Hence, the notions of semi-algebraic subsets and subsets definable in $\Ranexp$ of $\Qpv(\C)$ are well-defined.

\bigskip

In the following, we will only need the following facts.
\begin{itemize}
	\item Constructible subsets and morphisms of complex algebraic varieties are semi-algebraic. Semi-algebraic subsets and maps are definable in $\Ranexp$.
	\item The class of semi-algebraic (resp. definable in $\Ranexp$) subsets is stable under finite union, finite intersection, Cartesian product, complement, direct and inverse image by a semi-algebraic (resp. definable in $\Ranexp$) map.
	\item Peterzil-Starchenko's o-minimal Chow lemma: an analytic closed subset definable in $\Ranexp$ is Zariski closed (see \cite[Corollary 4.5]{peterzil2009complex}, that we can apply here because $\Ranexp$ is an o-minimal structure by \cite{van1994real}).
	\item The open subset $i(\Dom) \subset \Dual$ is semi-algebraic.
	\item The following theorem of Klingler, Ullmo and Yafaev.
\end{itemize}

\begin{thm}[see {\cite[Theorem 1.9]{klingler2016hyperbolic}}]\label{thm:kuy}
	Suppose the lattice $\Latt$ arithmetic. Then there exists a semi-algebraic subset $\Siegel \subset \Dom$ such that the restriction to $\Siegel$ of the quotient map $\Dom \to \Lsav$ is surjective and definable in $\Ranexp$.
\end{thm}

Let $x \in \Dom$, let $v \in T_{\Dom,x}$. Consider the vector space:
\[C(v) := \{w \in T_{\Dom,x}, [\iota_x(\bar{v}),\iota_x(w)] = 0 \} \subset T_{\Dom,x}. \]

\begin{rem}\label{rem:berg}
	Denote by $\tilde{\Ricu}$ the Riemann curvature tensor of the Bergman metric. It is seminegative in the sense of Griffiths. By \cite[(6.8.8) and (6.8.5)]{kobayashi2013hyperbolic}, we have $C(v) = \ker \tilde{\Ricu}(v,\bar{v})$.
\end{rem}

As $\iota_x$ is equivariant under the action of $\Isot_x$, for $g \in \Isot_{x,\C}$, we have:
\begin{equation}\label{eq:action}
	C(g.v) = \bar{g}.C(v).
\end{equation}
Denote:
\[\Charac_x^{\Gr} := \{C(v) \vert v \in T_{\Dom,x} \setminus \{0\}\} \subset \Gr(T_{\Dom,x})\]
with $\Gr(T_{\Dom,x}) := \bigsqcup_{k=0}^{\dim \Dom} \Gr_k(T_{\Dom,x})$ where $\Gr_k(T_{\Dom,x})$ is the Grassmannian variety of vector subspaces of dimension $k$ of $T_{\Dom,x}$.

\begin{lem}\label{lem:constr}
	The set $\Charac_x^{\Gr}$ is a finite union of orbits of $\Gr(T_{\Dom,x})$ under the algebraic action of the algebraic group $\Isot_{x,\C}$. In particular, it is a constructible subset of $\Gr(T_{\Dom,x})$.
\end{lem}

\begin{proof}
	The vector space $T_{\Dom,x}$ has finitely many orbits under the action of the algebraic group $\Isot_{x,\C} \times \C^*$, where $\Isot_{x,\C}$ acts by adjunction and $\C^*$ by homothety (see \cite[(1.1)]{mok2002characterization}). Hence, the first statement results from (\ref{eq:action}). Then the second statement results from Chevalley's theorem.
\end{proof}

Let $1 \leq p \leq \dim \Dom$. Denote:
\[\Charac_{x}^p := \bigcup_{C \in \Charac_x^{\Gr}} \{[\alpha] \vert \alpha \in \Lambda^p C \setminus \{0\}\} \subset \Proj \Omega_{\Dom,x}^{p}. \]
We use the following lemma to prove that $\Charac_{x}^p$ is closed.

\begin{lem}\label{lem:closed}
	Let $V,W$ be two vector spaces. Let $1 \leq p \leq \dim V$. Let $F$ be the subset of $\Proj \Lambda^p V^\vee \times \Hom(V,W)$ defined by:
	\[ F := \{([\alpha], f) \vert \alpha \in \Lambda^p \ker f \setminus \{0\} \}. \]
	Then $F$ is Zariski closed.
\end{lem}

\begin{proof}
	For $f \in \Hom(V,W)$, consider the linear map:
	\[ \app{\Lambda^p V}{W \otimes \Lambda^{p-1}V}{v_1 \wedge \dots \wedge v_p}{\sum_{i=1}^p (-1)^{i-1} f(v_i) \otimes v_1 \wedge \dots \wedge v_{i-1} \wedge v_{i+1} \wedge \dots \wedge v_p}. \]
	Its kernel is $\Lambda^p \ker f$. This procedure defines a linear map $\chi: \Hom(V,W) \to \Hom(V', W')$ with $V' := \Lambda^p V$ and $W' := W \otimes \Lambda^{p-1}V$. The following subset of $V' \times \Hom(V', W')$ is closed:
	\[G := \{(\beta, g) \vert g(\beta) = 0 \}.\]
	Pulling-back by $\chi$, we obtain that the following subset of $\Lambda^p V \times \Hom(V,W)$ is closed:
	\[F' := \{(\alpha, f) \vert \alpha \in \Lambda^p \ker f \}.\]
	Hence $F$ is closed in $\Proj \Lambda^p V^\vee \times \Hom(V,W)$.
\end{proof}

\begin{lem}\label{lem:charponczar}
	The set $\Charac_{x}^p$ is a Zariski closed subset of $\Proj \Omega_{\Dom,x}^{p}$. It is stable under the action of $\Isot_{x,\C}$.
\end{lem}

\begin{proof}
	Let $S \subset \Proj \Omega_{\Dom,x}^{p} \times \Gr(T_{\Dom,x})$ be defined as the set of points $([\alpha], C)$ such that $\alpha \in \Lambda^p C$ and $C \in \Charac_x^{\Gr}$. By \Cref{lem:constr}, it is a constructible subset stable under the action of $\Isot_{x,\C}$. Its projection on $\Proj \Omega_{\Dom,x}^{p}$ is $\Charac_{x}^p$. Hence, $\Charac_{x}^p$ is stable under the action of $\Isot_{x,\C}$ and, by Chevalley's theorem, it is constructible.
	
	Then, consider the Euclidean continuous map:
	\[\application{c}{T_{\Dom,x}}{\End_{\C}(T_{\Dom,x})}{v}{w \mapsto [\iota_x(\bar{v}),\iota_x(w)].}\]
	It induces a Euclidean continuous map $c': \Proj \Omega_{\Dom,x}^{p} \times T_{\Dom,x} \to \Proj \Omega_{\Dom,x}^{p} \times \End_{\C}(T_{\Dom,x})$. Now, let $S' \subset \Proj \Omega_{\Dom,x}^{p} \times T_{\Dom,x}$ be defined as the set of points $([\alpha], v)$ such that $\alpha \in C(v)$. By definition, it is the pull-back by $c'$ of the set $F$ of \Cref{lem:closed} (with $V = W = T_{\Dom,x}$). Hence, by this lemma, $S'$ is a Euclidean closed subset of $\Proj \Omega_{\Dom,x}^{p} \times T_{\Dom,x}$. As it is stable under the action of $\C^*$ on the second factor, it defines a Euclidean closed subset $S''$ of $\Proj \Omega_{\Dom,x}^{p} \times \Proj \Omega_{\Dom,x}^{1}$. By definition, $\Charac_{x}^p$ is the image of $S''$ by the projection on $\Proj \Omega_{\Dom,x}^{p}$. As $\Proj \Omega_{\Dom,x}^{1}$ is Euclidean compact, $\Charac_{x}^p$ is Euclidean closed. As $\Charac_{x}^p$ is both constructible and Euclidean closed, it is Zariski closed.
\end{proof}

If $\Anman$ is a complex manifold or an algebraic variety, we denote by $\Gr_k(T_{\Anman})$ the projective bundle on $\Anman$ that parametrizes vector subspace of dimension $k$ of every tangent spaces, and $\Gr(T_{\Anman}) := \bigsqcup_{k=0}^{\dim \Anman} \Gr_k(T_{\Anman})$. Denote by $\Charac_\Dom^p \subset \Proj \Omega_{\Dom}^p$ (resp. $\Charac_\Dom^{\Gr} \subset \Gr(T_\Dom)$) the disjoint union of each $\Charac_x^p$ (resp. $\Charac_x^{\Gr}$) for every $x \in \Dom$. Define $\Charac_{\Dual}^p$ as $H_\C.\Charac_{x}^p \subset \Proj \Omega_{\Dual}^p$ and $\Charac_{\Dual}^{\Gr}$ as $H_\C.\Charac_{x}^{\Gr} \subset \Gr(T_{\Dual})$. We have $\Charac_\Dom^p = \Charac_{\Dual}^p \cap \Proj \Omega_{\Dom}^p$ and $\Charac_\Dom^{\Gr} = \Charac_{\Dual}^{\Gr} \cap \Gr(T_{\Dom})$. Indeed, by \Cref{lem:charponczar} (resp. \Cref{lem:constr}), $\Charac_{x}^p$ (resp. $\Charac_{x}^{\Gr}$) is stable under the action of $\Isot_{x,\C}$, hence under the action of $P_x$.

\begin{lem}\label{lem:charancl}
	The subset $\Charac_{\Dual}^p$ is a Zariski closed subset and a holomorphic subbundle of $\Proj \Omega_{\Dual}^p$ over $\Dual$. In particular, the subset $\Charac_\Dom^p$ is a semi-algebraic subset and a closed holomorphic subbundle of $\Proj \Omega_{\Dom}^p$ over $\Dom$.

	The subset $\Charac_{\Dual}^{\Gr}$ is a constructible subset of $\Gr(T_{\Dual})$. In particular, the subset $\Charac_\Dom^{\Gr}$ is a semi-algebraic subset of $\Gr(T_\Dom)$.
\end{lem}

\begin{proof}
	Let $x \in \Dual$: we prove that in a neighborhood of $x$, $\Charac_{\Dual}^p$ is a closed holomorphic subbundle of $\Proj \Omega_{\Dual}^p$ over $\Dual$. By applying an element of $H_\C$, we can suppose $x \in \Dom$. Let $\plie_x$ be complex Lie subalgebra of $\hlie_\C$ defined by the complex Lie subgroup $P_x \subset H_\C$. Let $\mlie^+$ be a complement of $\plie_x$ in $\hlie_\C$. There exists an open neighborhood $U$ of $0$ in $\mlie^+$ such that $v \mapsto \exp(v).x$ is a biholomorphism from $U$ to an open neighborhood $V$ of $x$ in $\Dom$. This biholomorphism induces a local trivialization $\Proj \Omega_{\Dual,x}^p \times U \cong \Proj \Omega_{V}^p$ of the holomorphic bundle $\Proj \Omega_{\Dual}^p$ over $\Dual$. This trivialization restricts as a bijection $\Charac_x^p \times U \cong \Charac_{\Dual}^p \times_{\Dual} V$. As $\Charac_x^p$ is Zariski closed in $\Proj \Omega_{\Dual,x}^p$ by \Cref{lem:charponczar}, this proves that $\Charac_{\Dual}^p$ is a closed holomorphic subbundle of $\Proj \Omega_{\Dual}^p$.
	
	In particular, it is a closed analytic subset. Moreover, as $H_\C$ and $\Charac_x^p$ are algebraic varieties (see \Cref{lem:charponczar}) and $\Charac_x^{\Gr}$ is constructible (see \Cref{lem:constr}), $\Charac_{\Dual}^p$ and $\Charac_{\Dual}^{\Gr}$ are constructible by Chevalley's theorem. Hence $\Charac_{\Dual}^p$ is Zariski closed.
	
	As $i: \Dom \hookrightarrow \Dual$ is a holomorphic open embedding with semi-algebraic image, the same is true for the induced injection $\Proj \Omega_{\Dom}^p \hookrightarrow \Proj \Omega_{\Dual}^p$. Hence the statements on $\Charac_\Dom^p$ and $\Charac_\Dom^{\Gr}$ are immediate consequences of the foregoing.
\end{proof}

\begin{defi}\label{defi:hcsv}
	The \emph{higher degree characteristic subvariety} $\Charac^p \subset \Proj \Omega_{\Lsav}^p$ is the set of $[\alpha] \in \Proj \Omega_{\Lsav}^p$ such that there exist $x \in \Dom$, $v \in T_{\Dom,x} \setminus \{0\}$ and $\beta \in \Lambda^p T_{\Dom,x}$ such that $\alpha$ is the image of $\beta$ in $\Lambda^p T_{\Lsav}$ and such that $\beta \in \Lambda^p C(v)$.
\end{defi}

Recall that $C(v) = \{w \in T_{\Dom,x}, [\iota_x(\bar{v}),\iota_x(w)] = 0\}$ by definition and $C(v) = \ker \tilde{\Ricu}(v,\bar{v})$ by \Cref{rem:berg}. Denote by $\Charac^{\Gr}(\Lsav) \subset \Gr(T_\Lsav)$ (or $\Charac^{\Gr}$ if there is no ambiguity) the set of $C \in \Gr(T_\Lsav)$ such that there exist $x \in \Dom$ and $v \in T_{\Dom,x} \setminus \{0\}$ such that $C$ is the image of $C(v)$. Note that $\Charac^p$ (resp. $\Charac^{\Gr}$) is the quotient of $\Charac_\Dom^p$ (resp. $\Charac_\Dom^{\Gr}$) by the action of the fundamental group $\Latt$ of $\Lsav$.

\begin{rem}\label{rem:quotberg}
	Denote by $\Ricu$ the curvature of the canonical metric of $\Lsav$. Then $\Charac^p \subset \Proj \Omega_{\Lsav}^p$ is the set of $[\alpha] \in \Proj \Omega_{\Lsav}^p$ such that there exists $v \in T_{\Lsav}$ with $v \neq 0$ such that $\alpha \in \Lambda^p \ker \Ricu(v,\bar{v})$. In particular, this proves that $\Charac^1$ is the characteristic subvariety that Mok denotes by $\Charac_{r-1}$.
\end{rem}

\begin{lem}\label{lem:def}
	The subsets $\Charac^p \subset \Proj \Omega_{\Lsav}^p$ and $\Charac^{\Gr} \subset \Gr(T_\Lsav)$ are definable in $\Ranexp$.
\end{lem}

\begin{proof}
	The subset $\{([\alpha], C) \vert C \in \Gr(T_\Lsav), \alpha \in \Lambda^p C \setminus \{0\} \}$ of $\Proj \Omega_{\Lsav}^p \times_\Lsav \Gr(T_\Lsav)$ is Zariski closed. In particular, it is definable in $\Ranexp$. Thus the definability of $\Charac^p$ is a consequence of the definability of $\Charac^{\Gr}$.

	Let $\Lsav'$ and $\Lsav''$ be two locally symmetric varieties such that $\Lsav \cong \Lsav' \times \Lsav''$. Then, it follows from the definition of $\Charac^{\Gr}$:
	\begin{align*}
		\Charac^{\Gr}(\Lsav) = &\{C' \oplus C'' \vert C' \in \Charac^{\Gr}(\Lsav'), C'' \in \Charac^{\Gr}(\Lsav'')\}\\
		& \cup \{C' \oplus T_{\Lsav'', x} \vert C' \in \Charac^{\Gr}(\Lsav'), x \in \Lsav'' \}\\
		& \cup \{T_{\Lsav', x} \oplus C'' \vert x \in \Lsav', C'' \in \Charac^{\Gr}(\Lsav'') \}.
	\end{align*}
	Thus $\Charac^{\Gr}(\Lsav)$ is definable in $\Ranexp$ if and only if $\Charac^{\Gr}(\Lsav')$ and $\Charac^{\Gr}(\Lsav'')$ are definable. Hence, thanks to \Cref{lem:reduc}, we can suppose that $\Latt$ is irreducible.

	If $\Dom$ is a ball then $\Charac^{\Gr} = \Gr_0(T_\Lsav)$ so the result is obvious. Otherwise, $\Latt$ is arithmetic, so we can apply \Cref{thm:kuy}. Combined with the semi-algebraicity of $\Charac_{\Dom}^{\Gr}$ (see \Cref{lem:charancl}), we obtain that $\Charac_\Dom^{\Gr} \times_\Dom \Siegel$ is a semi-algebraic subset of $\Gr(T_\Dom) \times_\Dom \Siegel$, with $\Siegel$ defined in \Cref{thm:kuy}. By definition of $\Siegel$, the image of $\Charac_\Dom^{\Gr} \times_\Dom \Siegel$ in $\Gr(T_\Lsav)$ is definable in $\Ranexp$ and is equal to the image of $\Charac_\Dom^{\Gr}$. Hence, $\Charac^{\Gr}$ is definable in $\Ranexp$.
\end{proof}

Finally, the following proposition will be crucial to prove \Cref{thm:lsav}.

\begin{prop}\label{prop:charac}
	The subset $\Charac^p \subset \Proj \Omega_{\Lsav}^p$ is Zariski closed. It is a holomorphic subbundle of $\Proj \Omega_{\Lsav}^p$ on $\Lsav$.
\end{prop}

\begin{proof}
	The second statement and the analytic closure are immediate consequences of \Cref{lem:charancl}. By o-minimal Chow lemma, the first statement follows thanks to \Cref{lem:def}.
\end{proof}

\subsection{Augmented base locus of locally symmetric varieties}\label{subsection:lsavaug}

In this section, we prove \Cref{prop:lsavcharac} and \Cref{thm:lsav,thm:lsavbig}.

\begin{proof}[Proof of \Cref{prop:lsavcharac}]
	Let $\Latt' \subset \Latt$ be the finite index subgroup, $\Lsav'$ the finite étale cover of $\Lsav$, $(\bar{\Lsav}', \Div)$ the smooth projective logarithmic compactification and $\V_\R := V(\Hol(\Dom)^+, \hlie, \Latt')$ the $\R$-PVHS on $\Lsav'$ given by \Cref{prop:existspvhs}. By definition, $\V := \V_\R \otimes_\R \C$ has unipotent local monodromy on $\bar{\Lsav}'$.
	
	Denote by $p$ the quotient map $\Dom \to \Lsav'$. Let $x \in \Dom$, let $v,w \in T_{\Dom,x}$. Denote $y := p(x)$, $v' := \dif p_x.v$ and $w' := \dif p_x.w$. Let $\theta_{y}: T_{\Lsav,y} \to \End_{\C}(H_y)$ be the Higgs field of $\V$ at $y$, let $\theta_{y}^*$ be the conjugate Higgs field.
	
	The adjoint action of $\hlie$ on itself is faithful. Thus by \Cref{lem:higgslsav}, $[\theta_{y}^*(\bar{v}'),\theta_{y}(w')] = 0$ (in $\End_\C(\V_y)$) if and only if $[\iota_x(\bar{v}),\iota_x(w)] = 0$ (in $\hlie_\C$, notations of \Cref{subsection:zucker}). Hence we obtain the statement of the proposition by definition of $\Charac^p$ and of $C_{\V}$.
\end{proof}

In the case where $\Dom$ is irreducible and of rank at least $2$, the correspondence between $\Charac^p$ and $C_{\V}$ holds for every non-unitary $\C$-PVHS. Recall that the Hodge filtration of a $\C$-PVHS $\V$ is flat if and only if the Higgs field of $\V$ vanishes. In the case where the $\C$-PVHS is defined on a connected smooth complex quasiprojective variety, this is also equivalent to the relative compactness of the image of the monodromy representation of $\V$ in the linear group of a fiber (it is a consequence of \cite[Proposition 1.13]{deligne1987theoreme}). In this case, $\V$ is said to be \emph{unitary}.

\begin{prop}\label{prop:lsavcharac2}
	Let $\Dom$ be an irreducible bounded symmetric domain of rank at least $2$. Let $\Lsav$ be the quotient of $\Dom$ by the action of a torsion-free lattice.
	
	Then for every non-unitary $\C$-PVHS $\V$ on $\Lsav$ and every $1 \leq p \leq \dim \Lsav$, $\Charac^p$ is the set of $[\alpha] \in \Proj \Omega_{\Lsav}^p$ such that there exists $v \in T_{\Lsav,x} \setminus \{0\}$ such that $\alpha \in \Lambda^p C_{\V}(v)$ (see \Cref{nota:c}).
\end{prop}

\begin{proof}
	The domain $\Dom$ has rank at least $2$, moreover $\Dom$ is irreducible, so $\Latt$ is irreducible too. Thus we can apply \Cref{prop:classification}: for every $\C$-PVHS $\V$ on $\Lsav$, there exist a connected semisimple real Lie group $\Group$, a surjective morphism with compact kernel $\alpha: \Group \to \Hol(\Dom)^+$, a $\Group$-$\C$-HS $A$ and a torsion-free lattice $\Lambda' \subset \Group$ such that $\alpha(\Lambda')$ is a finite index subgroup of $\Latt$ and that the pull-back of $\V$ on $\Lambda' \backslash \Dom$ and $V(\Group, A, \Lambda')$ are isomorphic.

	As mentioned in \Cref{subsection:zucker}, there is a natural representation of the real Lie algebra $\hlie$ on $A$. As $\V$ is non-unitary, its Higgs field is not zero. Thus by \Cref{lem:higgslsav}, the action of $\hlie$ on $A$ is not trivial. As $\Dom$ is irreducible, $\hlie$ is simple. Hence its action on $A$ is faithful. Now we can conclude as in the proof of \Cref{prop:lsavcharac} just above.
\end{proof}

To apply Nakamaye's theorem in the proof of \Cref{thm:lsav}, we need the following lemma.

\begin{lem}\label{lem:nef}
	In \Cref{prop:existspvhs}, the tautological line bundle $\Tautp$ on $\Proj \Omega_{\bar{\Lsav}'}^p(\log \Div)$ is nef.
\end{lem}

\begin{proof}
	The log-Higgs field of $\V$ gives a locally split injection $\Lambda^p T_{\bar{\Lsav}'}(-\log \Div) \hookrightarrow \Lambda^p \End_{\Stru_{\bar{\Lsav}'}}(\ext{\VB})$. Hence it induces a locally split injection $\Stru_p(-1) \hookrightarrow \pi^*\Lambda^p \End_{\Stru_{\bar{\Lsav}'}}(\ext{\VB})$ with $\pi$ the projection $\Proj \Omega_{\bar{\Lsav}'}^p(\log \Div) \to \bar{\Lsav}'$. Because of \Cref{lem:vanish} (with $b = \id$ and $\LB = \Stru_p(-1)$), the Higgs field vanishes on $\Stru_p(-1)$. Now \Cref{thm:nef} gives our statement.
\end{proof}

\begin{proof}[Proof of \Cref{thm:lsav}]
	Because of \Cref{prop:fecover}, it is sufficient to prove \Cref{thm:lsav} up to a finite étale covering map. Hence, we can apply \Cref{prop:existspvhs} and suppose without loss of generality that $\Lsav' = \Lsav$. The inclusion $\AugBL(\Tautp) \cap \Proj \Omega_{\Lsav}^p \subset \Charac^p$ is given by \Cref{thm:main} and \Cref{prop:lsavcharac} applied to $\V$.
	
	Denote by $\Charac^{p,\reg}$ the smooth locus of $\Charac^p$. Let $x = [\alpha] \in \Charac^{p,\reg}$. Let $y := \pi(x)$ with $\pi$ the projection $\Proj \Omega_{\bar{\Lsav}}^p(\log \Div) \to \bar{\Lsav}$. By \Cref{prop:charac}, there exists a Euclidean open neighborhood $U$ of $y$ in $\Lsav$ and a biholomorphism $b: \Proj \Omega_{U}^p \cong \Proj \Omega_{\Lsav,y}^p \times U$ such that $b(\Charac^p \cap \Proj \Omega_{U}^p) = \Charac_y^p \times U$. This biholomorphism gives decompositions $T_{\Proj \Omega_{\Lsav}^p, x} \cong T_{\Proj \Omega_{\Lsav,y}^p, x} \oplus T_{\Lsav,y}$ and $T_{\Charac^{p,\reg},x} \cong T_{\Charac_y^{p},x} \oplus T_{\Lsav,y}$.
	
	Denote by $\came$ the canonical metric. On the one hand, consider formula (\ref{eq:cherntaut}): the kernel of the form $\omega := C_1(\Tautp,\came)_x$ on $T_{\Proj \Omega_{\Lsav}^p, x}$ is included in $T_{\Lsav,y}$. Hence, the dimension of this kernel stays the same after restriction to $\Charac^{p,\reg}$. In particular, $\omega_{\vert \Charac^{p,\reg}}$ is not definite positive at $x$ if and only if $\omega$ is not definite positive at $x$. By \Cref{lem:calcnkl}, this is equivalent to the existence of $v \in T_{\Lsav, y} \setminus \{0\}$ such that $\alpha \in \Lambda^p \ker \Ricu(v, \bar{v})$. The latter holds because $x \in \Charac^p$. Thus $\omega_{\vert \Charac^{p,\reg}}$ has not maximal rank at $x$. Therefore:
	\begin{equation}\label{eq:zero}
		\omega_{\vert \Charac^{p,\reg}}^{\dim \Charac^p} = 0.
	\end{equation}
	
	On the other hand, consider a desingularization $d: \hat{\Charac}^p \to \overline{\Charac}^p$, with $\overline{\Charac}^p$ the Zariski closure of $\Charac^p$ in $\Proj \Omega_{\bar{\Lsav}}^p(\log \Div)$. We have the equality:
	\begin{equation}\label{eq:cw}
		\int_{\hat{\Charac}^p} C_1(d^*\Tautp,d^*\came)^{\dim \Charac^p} = \int_{\hat{\Charac}^p} c_1(d^*\Tautp)^{\dim \Charac^p}
	\end{equation}
	If $\Lsav = \bar{\Lsav}$, then (\ref{eq:cw}) is given by Chern-Weil theory. In the general case, the integrability of $C_1(d^*\Tautp,d^*\came)^{\dim \Charac^p}$ and the equality (\ref{eq:cw}) are given by \cite[Theorem 5.20]{kollar1987subadditivity}, that extends Chern-Weil theory to some singular metrics that come from $\R$-PVHS. To apply it, we have to check that $d^*\Tautp$ is a subquotient of a Deligne extension of a real polarized variation of Hodge structure with unipotent local monodromy, and that $d^*\came$ is induced by the Hodge metric. This is given by \Cref{prop:existspvhs}: take the pull-back of the $\R$-PVHS $\Lambda^p\End_\R(\V_\R)$.
	
	Combining (\ref{eq:zero}) with (\ref{eq:cw}), we have $\int_{\hat{\Charac}^p} c_1(d^*\Tautp)^{\dim \Charac^p} = 0$, so $\int_{\overline{\Charac}^p} c_1(\Tautp)^{\dim \Charac^p} = 0$. By \Cref{lem:nef} $\Tautp$ is nef, so we can apply \Cref{thm:cas}, which gives $\Charac^p \subset \AugBL(\Tautp)$.
\end{proof}

\begin{proof}[Proof of \Cref{thm:lsavbig}]
	It is an immediate consequence of \Cref{thm:lsav} and \Cref{subsection:augpos}.
\end{proof}

\begin{rems}\label{rems:end}
	\begin{enumerate}
		\item The $l$ of \Cref{thm:lsavbig} is also the maximum dimension of boundary components of $\Dom$ (see \cite[Theorem 6.8.17]{kobayashi2013hyperbolic}). The values of $l$ are computed in \cite{wolf1965generalized}. If $\Dom$ is irreducible, the value of $l$ is given in the following table (that comes from \cite[p.329]{kobayashi2013hyperbolic}).\label{item:tab}
		\begin{center}
		\begin{tabular}{|c|c|c|c|}
			\hline
			& Domain & Dimension & $l$ \\
			\hline
			I &$ SU(p,q)/S(U(p) \times U(q))$ & $pq$ & $(p-1)(q-1)$ \\
			\hline
			II & $Sp(m,\R)/U(m)$ & $\frac{m(m+1)}{2}$ & $\frac{(m-1)m}{2}$ \\
			\hline
			III & $SO^*(2m)/U(m)$ & $\frac{m(m-1)}{2}$ & $\frac{(m-2)(m-3)}{2}$ \\
			\hline
			IV & $SO_0(m,2)/(SO(m) \times SO(2))$ & $m$ & $1$ \\
			\hline
			V & $E_6/(SO(10) \cdot SO(2))$ & $16$ & $1$ \\
			\hline
			VI & $E_7/(E_6 \cdot SO(2))$ & $27$ & $8$ \\
			\hline
		\end{tabular}
		\end{center}
		In general, $\Dom$ is a product $\prod_i \Dom_i$ of irreducible bounded symmetric domains. Then it follows from the definition of $l$ that:
		\[l(\Dom) = \max_i \left( l(\Dom_i) + \sum_{j \neq i} \dim \Dom_j \right).\]
		\item In \cite[Theorem 2.2]{sheng2010polarized}, Sheng and Zuo constructed a $\R$-PVHS on $\Lsav$ that verifies the hypotheses $(*)$ and $(**)$ of \Cref{prop:blrkcy}. Using the notations of this proposition, they prove that the set of $[v] \in \Proj \Omega_{\Lsav}^1$ above $x \in \Lsav$ such that $\theta_x^{i,j}(v) = 0$ is the characteristic variety $\Charac^1$. From there, \Cref{prop:blrkcy} gives the inclusion $\AugBL(\Stru_1(1)) \subset \Charac^1$. Note that Sheng and Zuo also obtain the other ``Mok's characteristic subvarieties'' of $\Proj \Omega_{\Lsav}^1$ by considering $\theta^{i,k}$ for $j \leq k \leq i$ (see \cite[Theorem 3.3]{sheng2010polarized}).
		\item As a consequence of \Cref{thm:lsav}, we have that for every $x \in \Lsav$ and for every $[\alpha] \in \AugBL(\Tautp)_x$, there exists a vector subspace $E \subset T_{\Lsav,x}$ such that $\alpha \in \Lambda^p E$ and the subset $\Proj E^\vee$ of $\Proj \Omega_{\Lsav,x}^1$ is included in $\AugBL(\Stru_1(1))$.
		\item In the computation of the augmented base loci, it is possible not to use \Cref{thm:main} and to use the canonical metric instead of the Hodge metric (see \cite[Chapter 6.8]{kobayashi2013hyperbolic}). However, for the non-compact case, there are two points where we do not know how to avoid the use of variational Hodge theory: to prove the nefness of the tautological line bundle, and to use Chern-Weil theory in the proof of \Cref{thm:lsav}.\label{item:nopvhs}
	\end{enumerate}
\end{rems}

\appendix

\section{SageMath code}\label{section:appendix}

\begin{lstlisting}
from sage.all import *
from itertools import product

degrees5 = [vector(t) for t in product(range(4), repeat=5) if sum(t) == 5]
# List of multidegrees of monomials generating R5

degrees10 = [vector(t) for t in product(range(4), repeat=5) if sum(t) == 10]
# List of multidegrees of monomials generating R10

index_map = {tuple(v): i for i, v in enumerate(degrees5)}
# Dictionary to find the index of a monomial in degrees5 from its multidegree

S = PolynomialRing(QQ, 101, 'x')

def coordProd(i):
    """
    Computes one of the coordinates of the product R5 x R5 -> R10.

    Parameters
    ----------
    i : int
        Index of the coordinate (between 0 and 100).

    Returns
    -------
    res : polynomial
        Quadratic form giving the i-th coordinate of the bilinear map R5 x R5 -> R10.
    """
    
    res = S.zero()
    monom10 = degrees10[i]
    for j in range(101):
        monom5 = degrees5[j]
        diff = monom10 - monom5
        if min(diff) >= 0:
            k = index_map.get(tuple(diff))
            if k is not None:
                res += S.gen(j) * S.gen(k)
    return res

chosenCoord = [24, 10, 45, 89, 50, 63, 9, 22, 58, 78, 29, 82, 1, 0, 2, 3, 4, 7, 8, 100, 99, 98]
# The calculation time depends heavily on the choice of the coordinates. This list is the best that we have found for performing calculations in a reasonable amount of time (2-3 minutes on our computer. A few hours were not enough to compute the dimension with one more coordinate).

quadForms = [coordProd(i) for i in chosenCoord]

I = S.ideal(quadForms)
# Ideal included in the ideal generated by all the coordProd(i).

print(I.dimension())
# Prints "79", giving an upper bound to the dimension of the variety defined by the vanishing of the square R5 -> R10.
\end{lstlisting}

\printbibliography

\end{document}